\documentclass[12 pt]{article}
\usepackage{amsmath}
\usepackage{amsfonts}
\usepackage{amssymb}
\usepackage{amsthm}
\usepackage{thmtools}
\usepackage{hyperref}
\usepackage{mathtools}
\usepackage{enumitem}
\usepackage{graphicx}
\usepackage{subcaption}
\usepackage[style=numeric,backend=bibtex]{biblatex}
\newtheorem{thm}{Theorem}[section]
\newtheorem{lem}[thm]{Lemma}

\newtheorem{cor}[thm]{Corollary}
\newtheorem{rmk}[thm]{Remark}
\begin{document}
\title{Knot Floer homology of some even 3-stranded pretzel knots}
\author{Konstantinos Varvarezos}
\maketitle

\begin{abstract}
We apply the theory of ``peculiar modules" for the Floer homology of 4-ended tangles developed by Zibrowius \cite{Zib} (specifically, the immersed curve interpretation of the tangle invariants) to compute the Knot Floer Homology ($\widehat{HFK}$) of 3-stranded pretzel knots of the form ${P(2a,-2b-1,\pm(2c+1))}$ for positive integers $a,b,c$.  This corrects a previous computation by Eftekhary \cite{Eft}; in particular, for the case of ${P(2a,-2b-1,2c+1)}$ where $b<c$ and $b<a-1$, it turns out the rank of $\widehat{HFK}$ is larger than that predicted by that work.
\end{abstract}

\section{Introduction}

\begin{figure}
\centering
\begin{subfigure}{.45\textwidth}
\centering
\includegraphics[width=\textwidth]{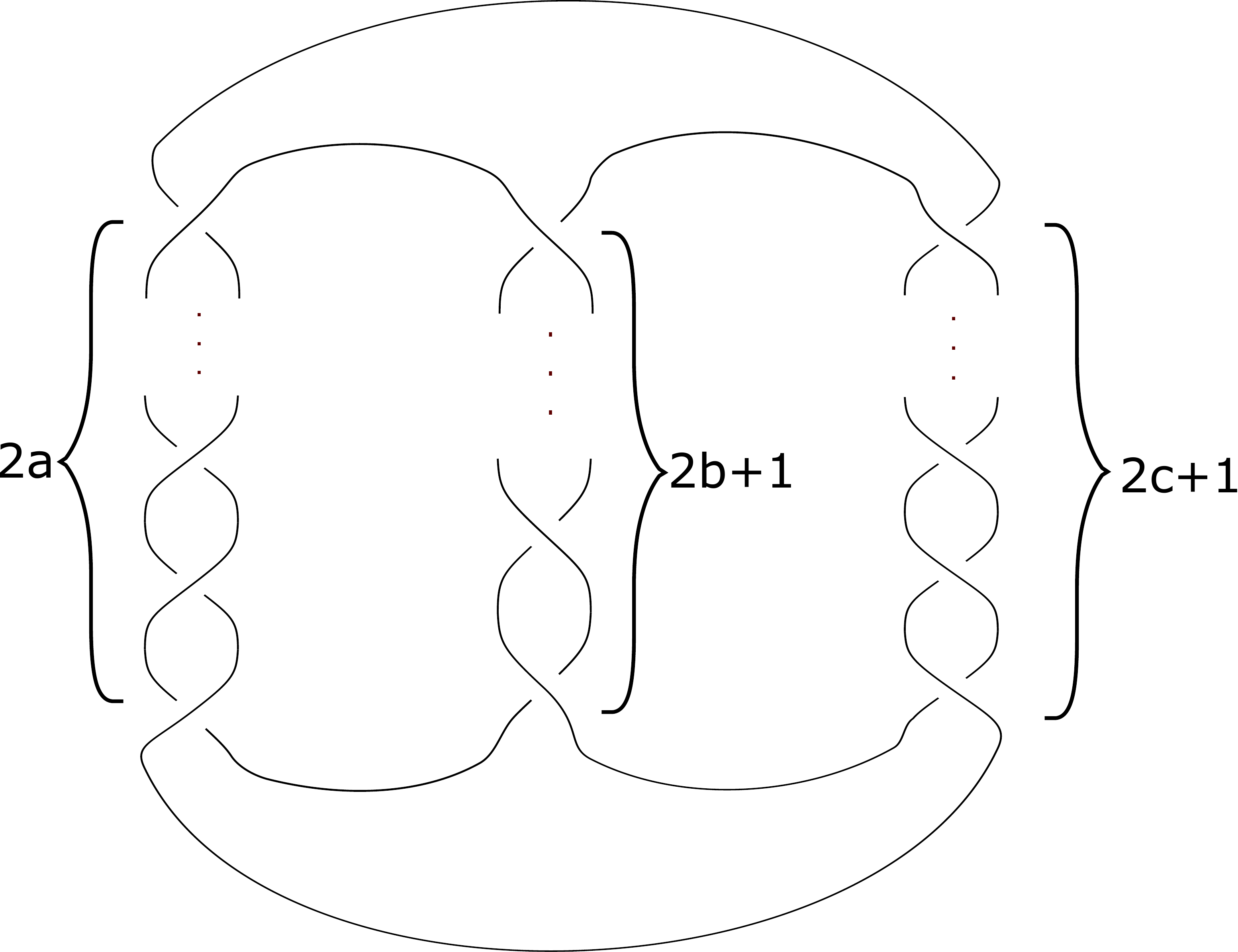}
\caption{The pretzel knot ${P(2a,-2b-1,-2c-1)}$}	
\end{subfigure}
\hfill
\begin{subfigure}{.45\textwidth}
\centering
\includegraphics[width=\textwidth]{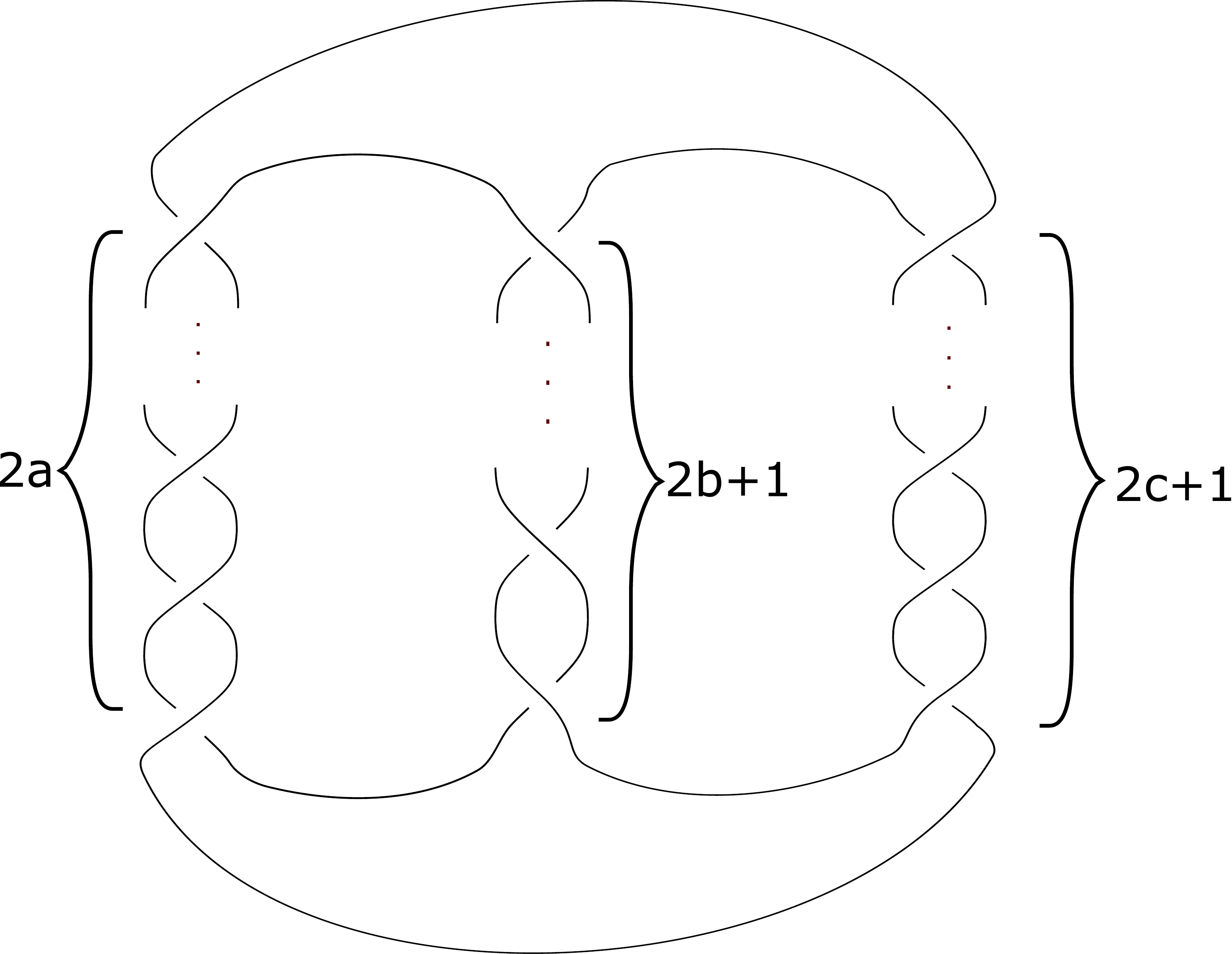}
\caption{The pretzel knot ${P(2a,-2b-1,2c+1)}$}
\end{subfigure}
\caption{The two types of pretzel knots we shall consider.}
\label{fig:pretz}
\end{figure}

The aim of this work is to compute the Knot Floer homology (specifically $\widehat{HFK}$) of a certain family of 3-stranded pretzel knots.  These are, in particular, pretzel knots of the form $P(2a,-2b-1,-2c-1)$ and $P(2a,-2b-1,2c+1)$ for $a,b,c>0$; see Figure \ref{fig:pretz} for an illustration.  Knot Floer Homology for other kinds of 3-stranded pretzel knots is already known.  When all strands have the same ``sign", these are alternating, and so $\widehat{HFK}$ is essentially determined by the Alexander polynomial.  The non-alternating case with all odd twists was computed in \cite{OSzP}.  Our case (non-alternating, even pretzel knots) was considered by Eftekhary in \cite{Eft}.  However, an error in that result was found by Manion; in particular Theorem 2 (the case $P(2a,-2b-1,2c+1)$) does not always hold when $b<\min(a,c)$.  Moreover, Manion, as related in a personal communication with the author, found that the result could be corrected as follows: it turns out that there were exactly two ``exceptional" Alexander gradings where Eftekhary's result fails (in particular, the homology lies on two delta gradings instead of one at those Alexander gradings), and using that combined with fact that for these knots, delta-graded Knot Floer homology is the same as delta-graded (reduced) Khovanov homology, one could deduce from \cite{Man} what the ``missing" ranks had to be.  More recently, Waite in his PhD thesis \cite{Wai} has computed the Knot Floer complex for our family of knots using the bordered algebra techniques from \cite{OSzBord}.  The results in this work agree with Manion's description as well as Waite's results (in particular, cf. his Theorem 6.3), although we use different methods.

The rest of this work is organized as follows: in the next section, we give a brief outline of Knot Floer homology and the tangle invariant $HFT$ introduced by Zibrowius \cite{Zib}, and we describe how to obtain Knot Floer homology by pairing tangle invariants.  We conclude that section by enumerating the different tangle invariants that will arise in our computations for the pretzel knots we are interested in.  Each of the two sections afterwards details the pairing computations involved in the calculation of the Knot Floer homology of the families $P(2a,-2b-1,-2c-1)$ and $P(2a,-2b-1,2c+1)$ respectively.

\subsection{Summary of results}
Here we give a brief overview of our main results.  Firstly, in the case $K=P(2a,-2b-1,-2c-1)$, we shall see in the proof of Theorem \ref{thm1} that $\widehat{HFK}(K)$ can have one of two different forms: if $a>b$ or $a>c$, all generators are supported in a single delta grading (this is also called being \textit{homologically thin}); if $a\leq b$ and $a\leq c$, then the support is in two consecutive delta gradings $\delta_1 <\delta_2$.  In this latter situation, the generators with delta grading $\delta_1$ will have Alexander grading at least two more than (or at most two less than) the maximal (respectively, minimal) Alexander grading of any generator with delta grading $\delta_2$.  In either case, for any Alexander grading, the generators lie in at most one delta grading; it follows that $\widehat{HFK}(K)$ is determined by the Alexander polynomial for these knots. See Figure \ref{fig:HFK_1} for illustrative examples of the two possible cases; for these graphs, $\mu$ is the Maslov grading, $s$ is the Alexander grading, and the delta grading is the difference $s-\mu$ and so lines of constant delta grading are diagonal. A solid dot indicates nonzero rank of $\widehat{HFK}(K)$ with the appropriate gradings.
\begin{figure}
\centering
\begin{subfigure}{.45\textwidth}
\centering
\includegraphics[width=\textwidth]{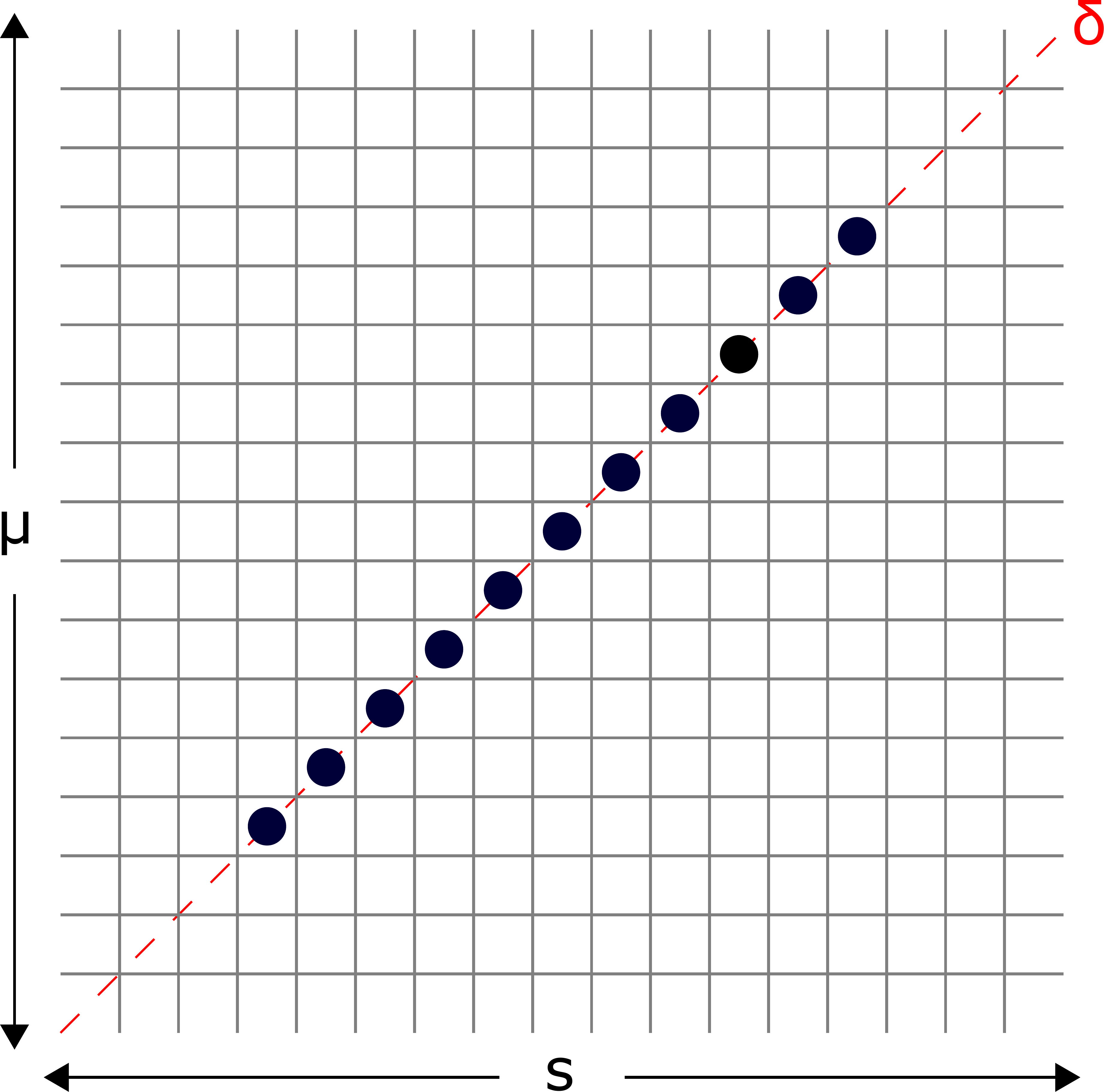}
\caption{If $a>b$ or $a>c$, then $\widehat{HFK}(K)$ is thin}	
\end{subfigure}
\hfill
\begin{subfigure}{.45\textwidth}
\centering
\includegraphics[width=\textwidth]{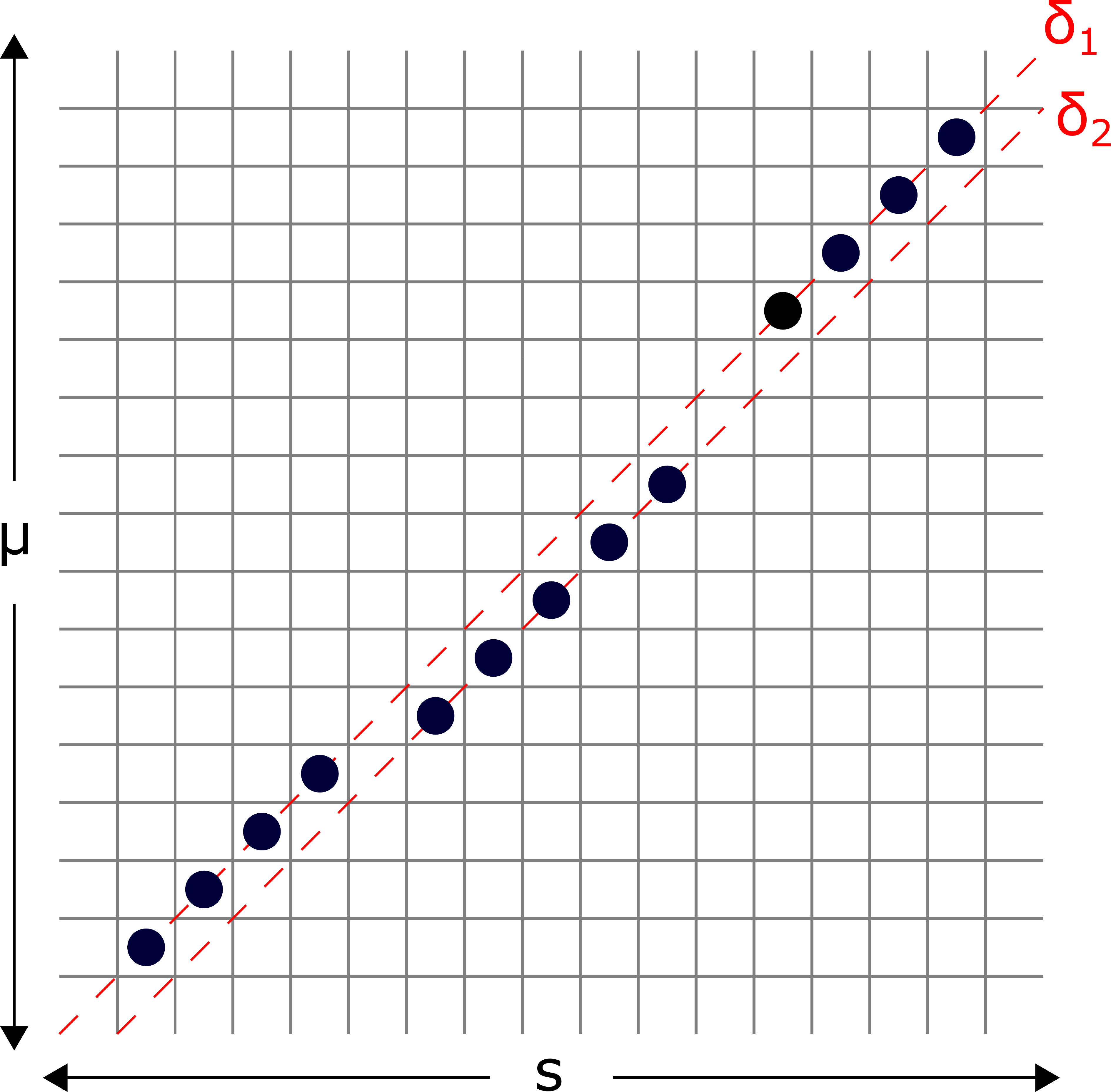}
\caption{If $a\leq b$ and $a\leq c$, then $\widehat{HFK}(K)$ is supported in two delta gradings with no overlap for any Alexander grading}
\end{subfigure}
\caption{The two kinds of forms the Knot Floer homology of $K={P(2a,-2b-1,-2c-1)}$ can take}
\label{fig:HFK_1}
\end{figure}

\begin{figure}
\centering
\begin{subfigure}{.45\textwidth}
\centering
\includegraphics[width=\textwidth]{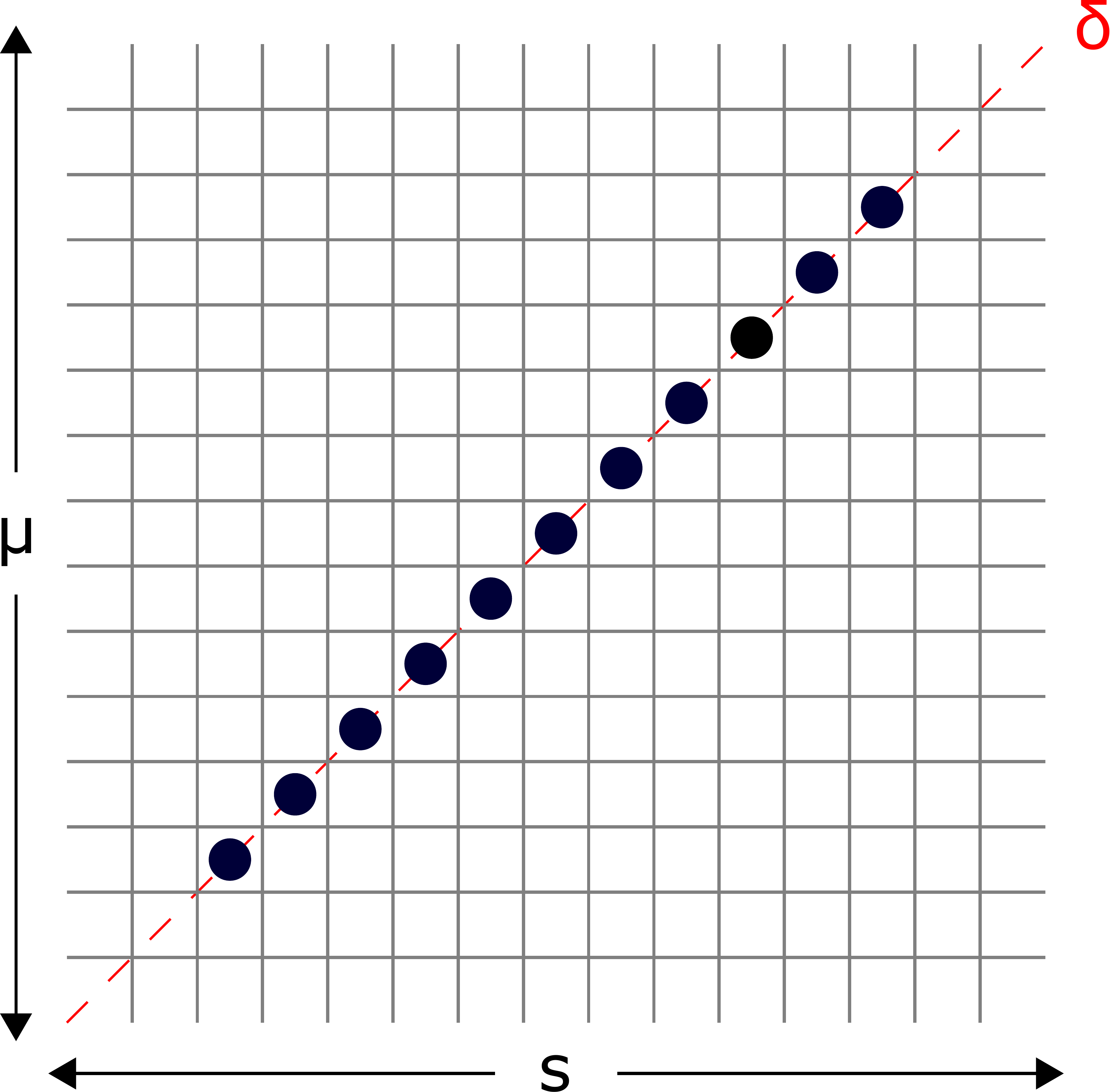}
\caption{If $a\leq b$ or $c\leq b$, then $\widehat{HFK}(K)$ is thin}
\end{subfigure}
\hfill
\begin{subfigure}{.45\textwidth}
\centering
\includegraphics[width=\textwidth]{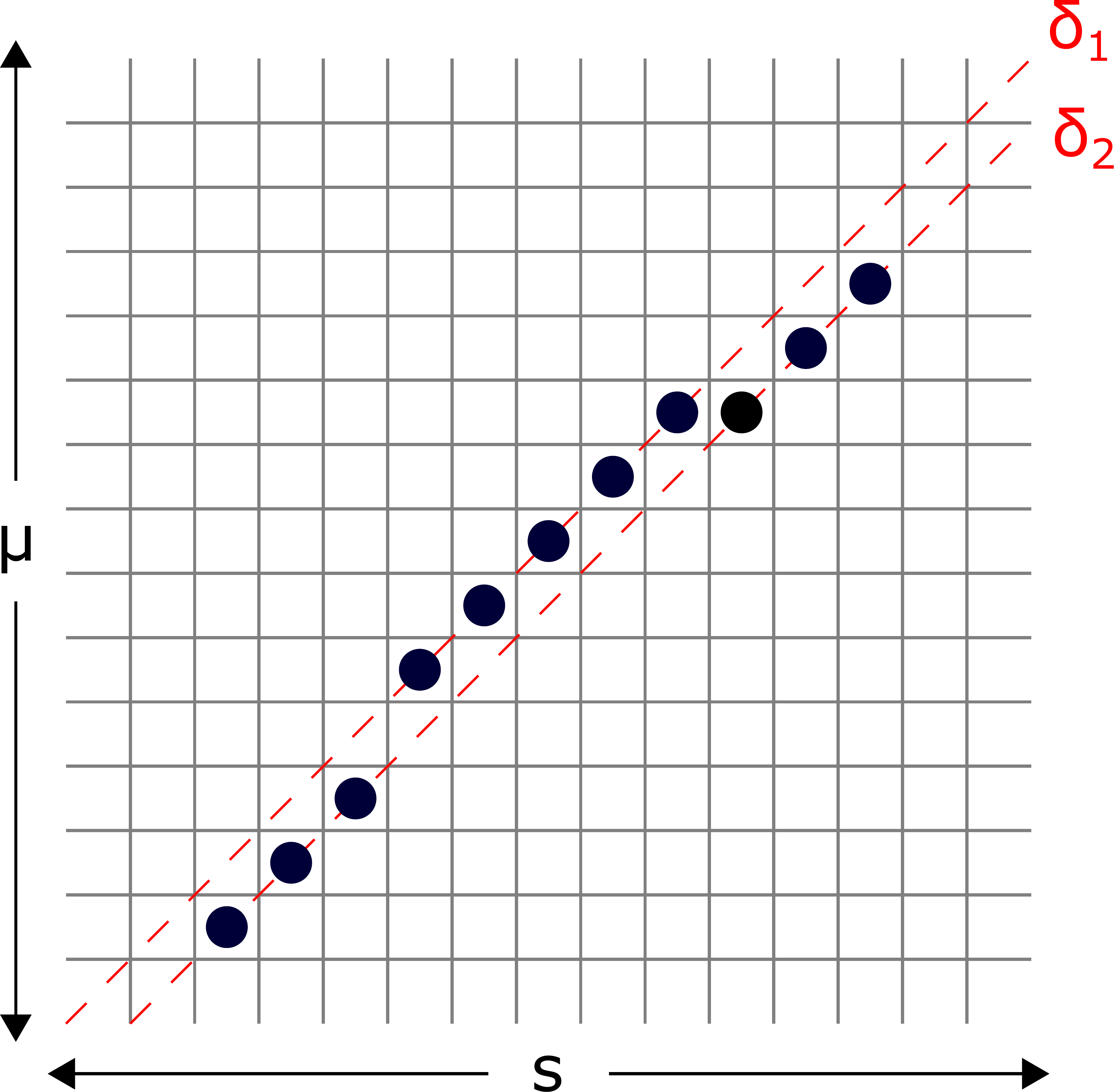}
\caption{If $a=b+1$ and $c\leq b$, then $\widehat{HFK}(K)$ is supported in two delta gradings with no overlap for any Alexander grading}
\end{subfigure}
\\
\begin{subfigure}{.45\textwidth}
\centering
\includegraphics[width=\textwidth]{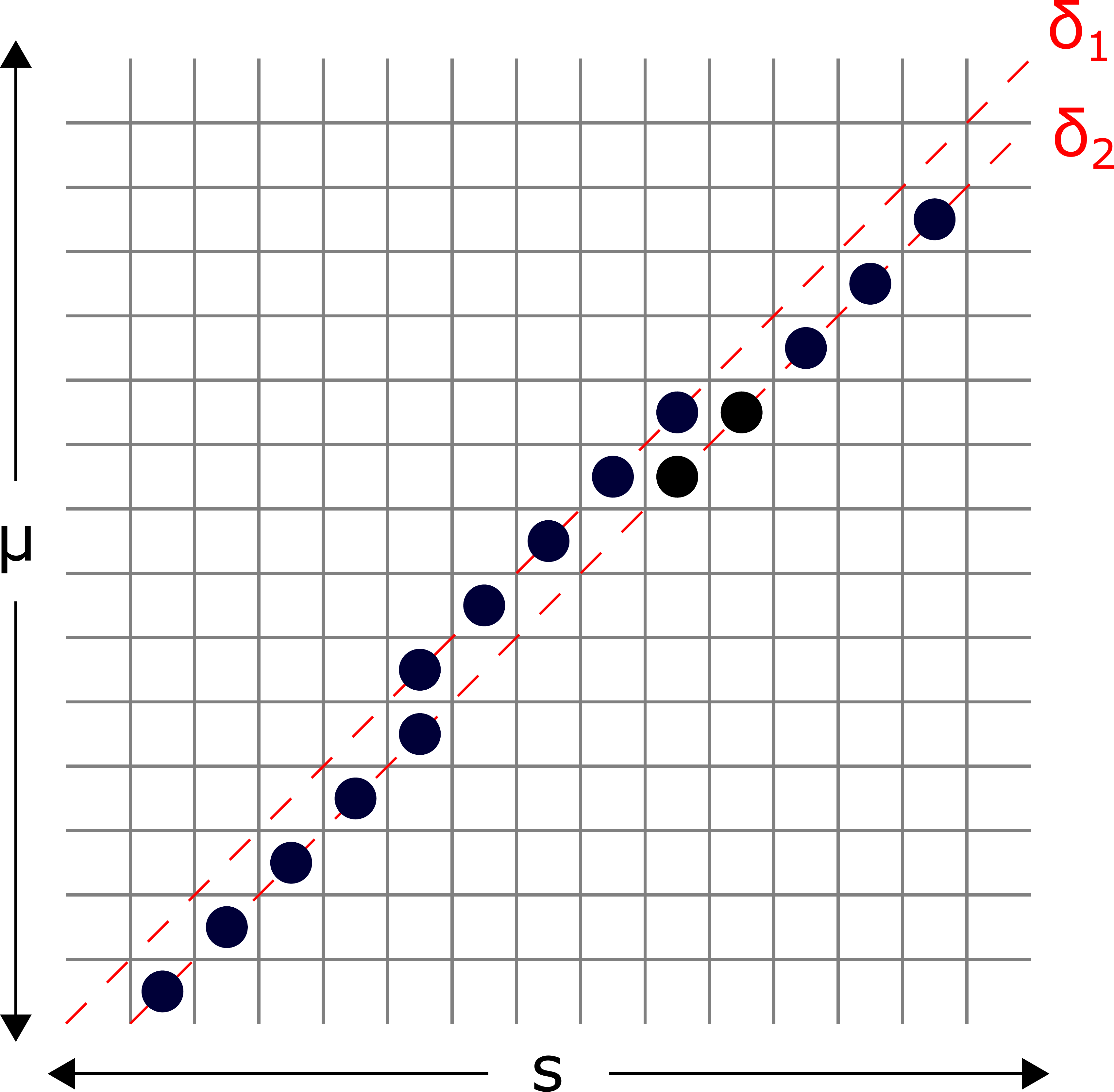}
\caption{If $a>b+1$ and $b<c$, then $\widehat{HFK}(K)$ is supported in two delta gradings with overlap occurring at exactly one pair of Alexander gradings}
\end{subfigure}
\caption{The three kinds of forms Knot Floer homology of ${P(2a,-2b-1,2c+1)}$ can take}
\label{fig:HFK_2}
\end{figure}

Now, in the case $K=P(2a,-2b-1,-2c-1)$, we shall see in the process of proving Theorem \ref{thm2} that there are three main possibilities.  If $a\leq b$ or $c\leq b$, then $\widehat{HFK}(K$ is thin.  Otherwise, it is supported in two consecutive delta gradings $\delta_1<\delta_2$.  There are two further subcases: if $a=b+1$ and $c\leq b$, then the generators with delta grading $\delta_2$ all have Alexander grading at least one more than the maximal (or at most one less than the minimal) Alexander grading of any generator with delta grading $\delta_1$.  Therefore, in this case, as well as in the thin case, the Knot Floer homology is determined by the Alexander polynomial.  The final subcase is when $a>b+1$ and $c>b$.  In this case, however, ``overlap" occurs between the two delta ``lines"; more precisely, at the maximal and minimal Alexander gradings of the generators with delta grading $\delta_1$, there will be nonzero generators with delta grading $\delta_2$, but these are the only Alexander gradings for which this occurs.  Hence, in this case, the rank of the Knot Floer homology for this pair of Alexander gradings is greater than that predicted by the Alexander polynomial.  In particular, this is the case for which Eftekhary's result \cite{Eft} needs correction.  See Figure \ref{fig:HFK_2} for illustrative examples of the three kinds of structures Knot Floer homology can have for these.

\section{Preliminaries}
\subsection{Knot Floer homology and tangle invariants}
Knot Floer homology is a knot invariant introduced by Ozsv\'{a}th and Szab\'{o} \cite{OSzHFK} and independently by J. Rasmussen \cite{Ras}. We briefly recall some of the basic properties of this invariant. Given a knot $K$, we shall be interested in the so-called ``hat" version of Knot Floer homology of $K$, denoted $\widehat{HFK}(K)$. This takes the form of a bigraded vector space over a field $\mathbf{k}$ (for convenience, we will take $\mathbf{k}=\mathbf{Z}/2\mathbf{Z}$, but this should not affect the results).  The two usual gradings are the \textit{Alexander} and \textit{Maslov} gradings.  However, instead of the Maslov grading we shall use the \textit{delta} grading, defined as the difference between the Alexander and Maslov gradings.  The subspace of elements of Alexander grading $s$ and delta grading $\delta$ will be denoted: $\widehat{HFK}_{\delta}(K,s)$.  It is well-known that the classical Alexander polynomial of a knot $K$ can be recovered by taking the graded Euler characteristic of $\widehat{HFK}(K)$:
\[
\Delta_K(t) = \sum_{s,\delta}(-1)^{\mu(s,\delta)}\mathrm{rk}\left(\widehat{HFK}_{\delta}(K,s)\right)t^s
\]
where $\mu(s,\delta)=s-\delta$ is the Maslov grading corresponding to Alexander grading $s$ and delta grading $\delta$.  Conversely, we see that if, for a given Alexander grading $s$, $\widehat{HFK}(K,s)$ is supported in a single delta grading, then the rank of $\widehat{HFK}(K,s)$ is determined by the corresponding coefficient in the Alexander polynomial.

Our computation of the Knot Floer homology will be accomplished using a Heegaard Floer invariant for 4-ended tangles developed by Zibrowius \cite{Zib}.  For our purposes, we shall be interested in the invariant $HFT(T)$ which, for an oriented 4-ended tangle $T$, consists of a collection of (graded) immersed curves on the parametrized 4-punctured sphere $S^2_4$.  The parametrization consists of four arcs cyclically connecting the punctures; we shall depict these as edges of a ``square" which we refer to as the \textit{parametrizing square}.  The ``grading" of a curve consists of Alexander and delta gradings on each of its intersection points with the parametrizing square. (Technically, the curve should also be equipped with a ``local system", which is a matrix $X\in \mathrm{GL}_n(\mathbf{k})$; however, all of the curve invariants that we will need have a trivial local system, so we will suppress this). Moreover, each punctured has an associated sign, $+1$ if the tangle is ``inwardly pointing" at that puncture and $-1$ if it is ``outwardly pointing".  In our case, the punctures on top will be negative and those on the bottom will be positive.  Given two tangles $T_1$ and $T_2$ (appropriately oriented), we denote by their \textit{pairing} the link obtained by gluing their endpoints together as indicated by Figure \ref{fig:pair}.  The \textit{mirror} of a tangle $T$, denoted $\mathrm{mr}(T)$, is obtained by reflecting the tangle across a plane (in a diagram, this may be obtained by reversing all the crossings) and also reversing the orientations of the strands.
\begin{figure}
\centering
\includegraphics[width=.5\textwidth]{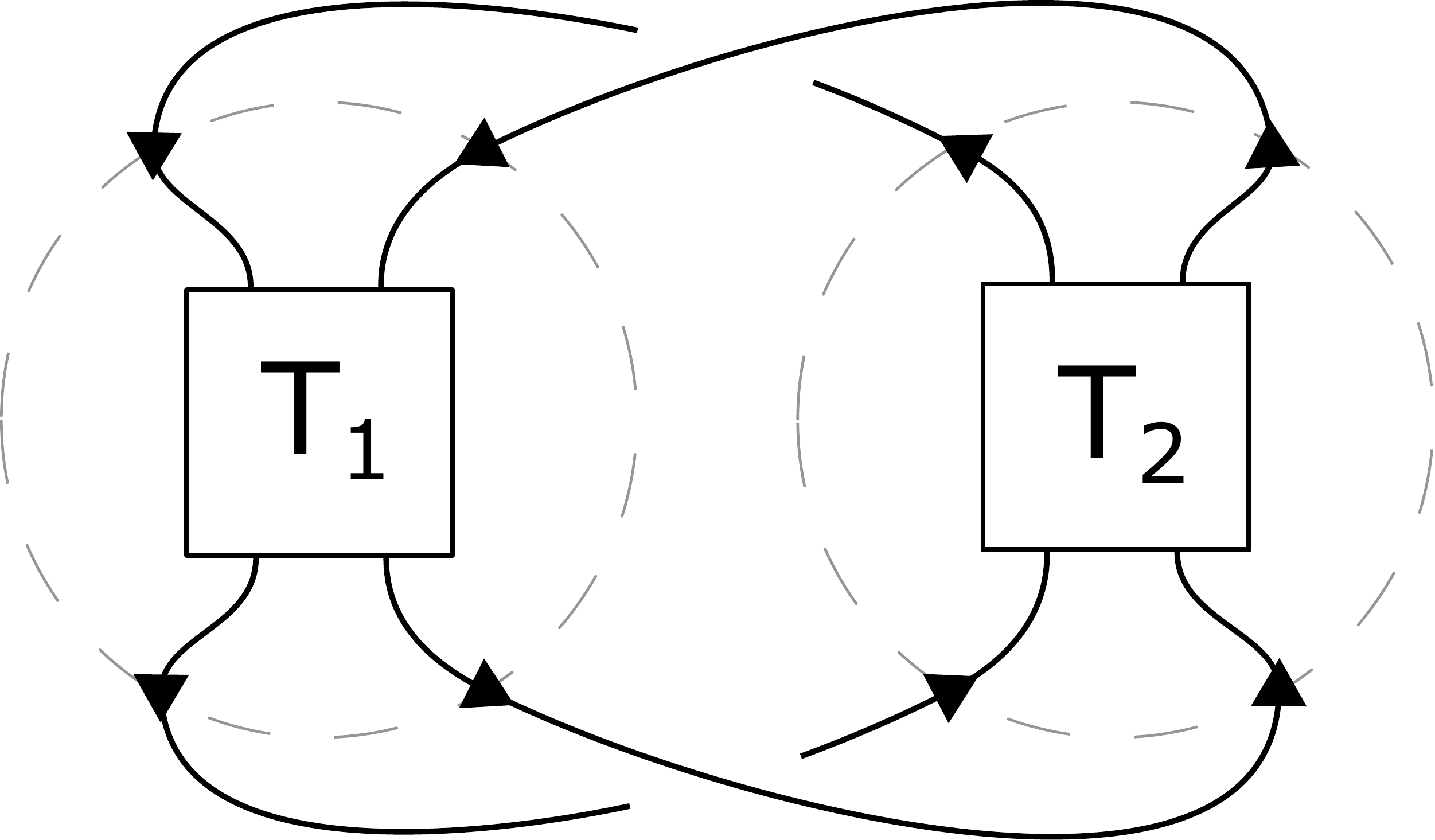}
\caption{The link obtained by pairing the oriented tangles $T_1$ and $T_2$}
\label{fig:pair}
\end{figure}

The following key result will be our main tool of use:
\begin{thm}[Theorem 5.9 of \cite{Zib}]\label{thm:pair}
Let $T_1$ and $T_2$ be (appropriately oriented) 4-ended tangles, and suppose that their pairing results in a knot $K$. Then
\[
\widehat{HFK}(K) \otimes V \cong HF(HFT(\mathrm{mr}(T_1),HFT(T_2))
\]
where $V$ is the 2-dimensional vector space supported in Alexander gradings $\pm 1$ and identical delta gradings (which we may take to be 0).
\end{thm}
Here $HF$ represents the (Lagrangian) Floer homology and $HFT$ is the immersed curve tangle invariant associated to the peculiar modules.  In practice, the Floer homology is found by placing the curves in minimal-intersection position.
\begin{rmk}\label{rmk:half}
The isomorphism in Theorem \ref{thm:pair} preserves delta grading up to an overall shift.  The (symmetrized) Alexander gradings are preserved only after they are first doubled in $\widehat{HFK}(K)$; that is, a generator of $\widehat{HFK}(K)$ with Alexander grading $s$ will be mapped under the isomorphism, to a pair of generators with Alexander gradings $2s\pm 1$.
\end{rmk}

\begin{figure}
\centering
\begin{subfigure}{.3\textwidth}
\centering
\includegraphics[width=\textwidth]{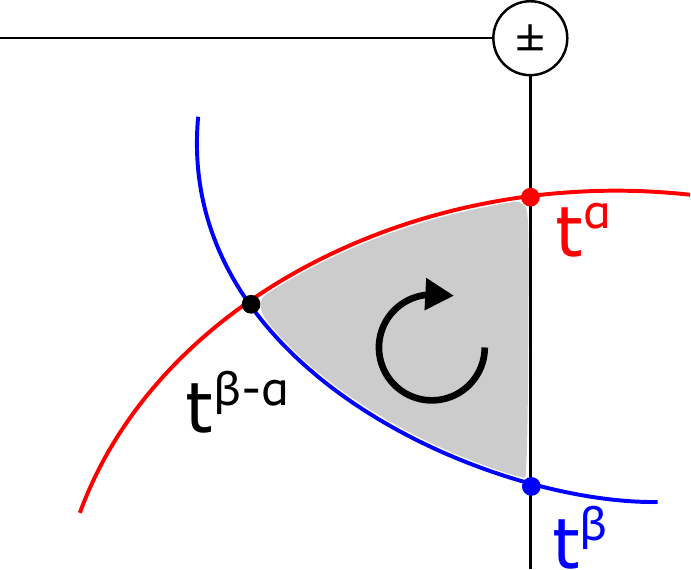}
\caption{The case where the ``disk" does not cover any puncture}
\end{subfigure}
\hfill
\begin{subfigure}{.3\textwidth}
\centering
\includegraphics[width=\textwidth]{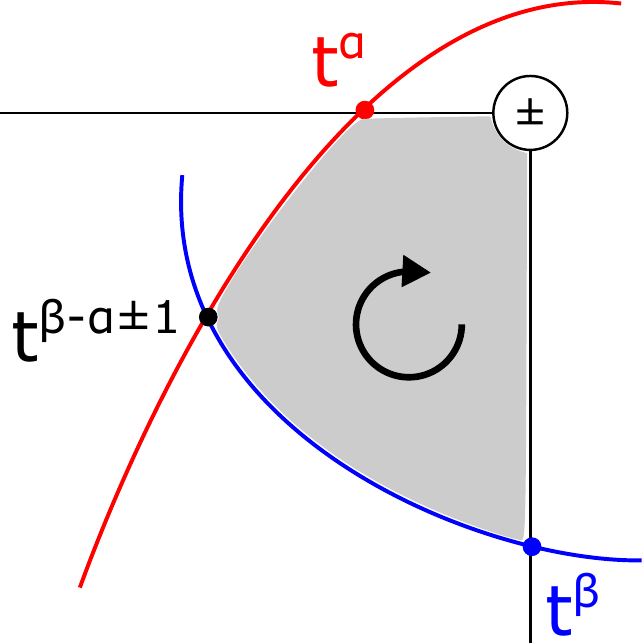}
\caption{The case where the ``disk" covers one puncture}
\end{subfigure}
\hfill
\begin{subfigure}{.3\textwidth}
\centering
\includegraphics[width=\textwidth]{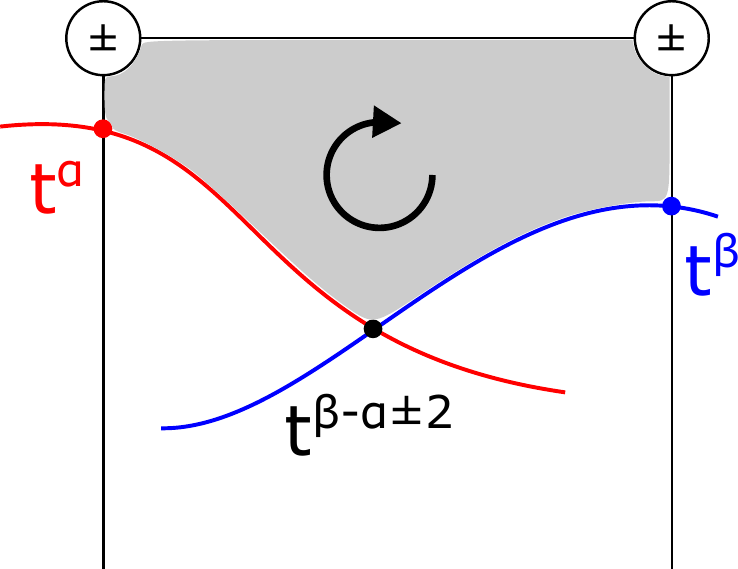}
\caption{The case where the ``disk" covers two punctures (of the same sign)}
\end{subfigure}
\caption{Three typical cases arising in the computation of the Alexander gradings of intersection points}
\label{fig:int_A}
\end{figure}

\begin{figure}
\centering
\begin{subfigure}{.3\textwidth}
\centering
\includegraphics[width=\textwidth]{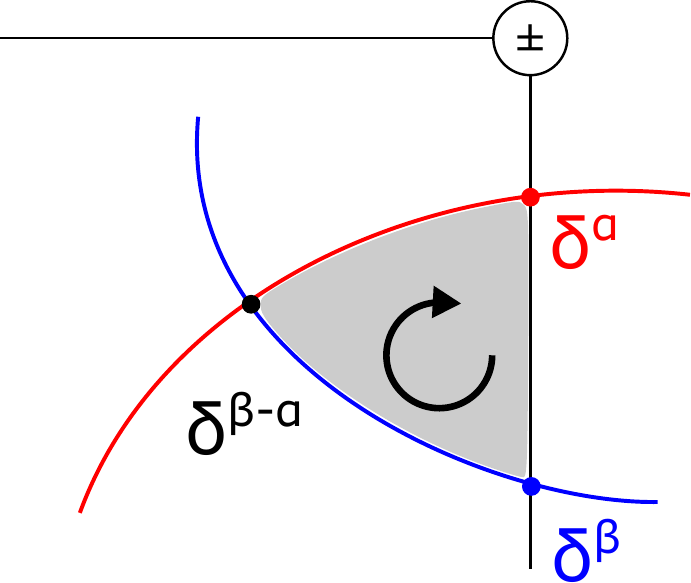}
\caption{The case where the ``disk" does not cover any puncture}
\end{subfigure}
\hfill
\begin{subfigure}{.3\textwidth}
\centering
\includegraphics[width=\textwidth]{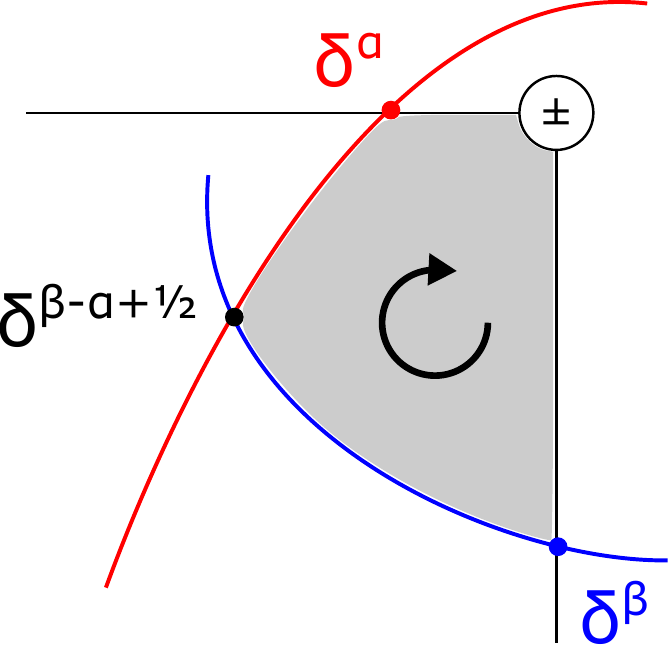}
\caption{The case where the ``disk" covers one puncture}
\end{subfigure}
\hfill
\begin{subfigure}{.3\textwidth}
\centering
\includegraphics[width=\textwidth]{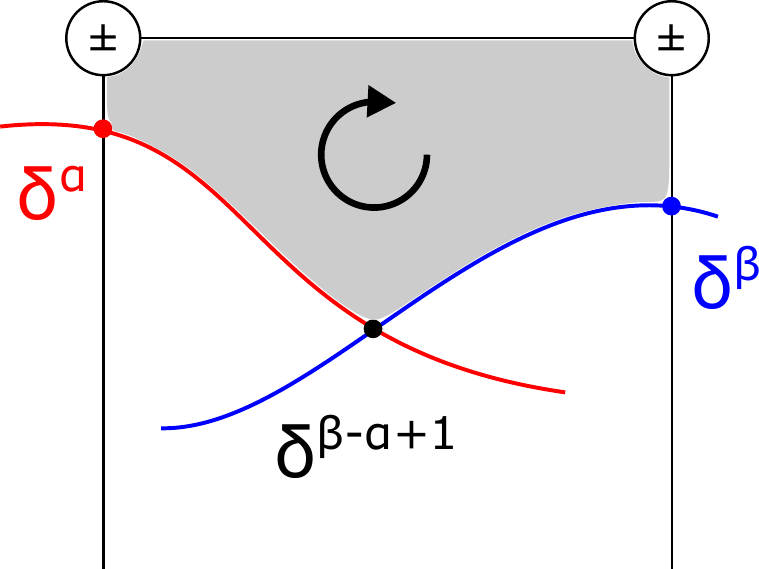}
\caption{The case where the ``disk" covers two punctures}
\end{subfigure}
\caption{Three typical cases arising in the computation of the delta gradings of intersection points}
\label{fig:int_D}
\end{figure}

With the curves in a minimal-intersection position, the intersection points correspond to the generators of the Floer homology.  The Alexander and delta gradings of these generators may be computed as follows.  For each intersection point $z$, one finds a ``disk" $\phi$ in the sphere with boundary lying on both the ``red" and ``blue" curves as well as the parametrizing square.  Let us denote the intersections bewteen the square and the red and blue curves by $x$ and $y$, respectively.  This disk should also be such that when traversing its boundary \emph{clockwise}, the blue curve is traversed from $y$ to $z$, followed by the red curve being traversed from $z$ to $x$.  Further, let $P(\phi)$ denote the set of punctures covered by the disk $\phi$.  Then, the Alexander grading of the generator is given by:
\begin{equation}\label{eq:int_A}
\hat{A}(z)=\hat{A}(y)-\hat{A}(x) + \sum_{p\in P(\phi)}\varepsilon_p
\end{equation}
where $\varepsilon_p$ denotes the sign of the puncture (+1 for an ``incoming" strand and -1 for an ``outgoing" strand).  See Figure \ref{fig:int_A} for an illustration of three common cases.
Similarly, for the delta grading, we have:
\begin{equation}\label{eq:int_D}
\delta(z)=\delta(y)-\delta(x) + \sum_{p\in P(\phi)}\frac{1}{2}
\end{equation}
See Figure \ref{fig:int_D} for an illustration of such calculations; note that we shall often use $\delta^{\beta}t^{\alpha}$ to denote an intersection point/ generator of Floer homology with Alexander grading $\alpha$ and delta grading $beta$, although we may suppress one of the two for convenience.
\begin{figure}
\centering
\includegraphics[width=.5\textwidth]{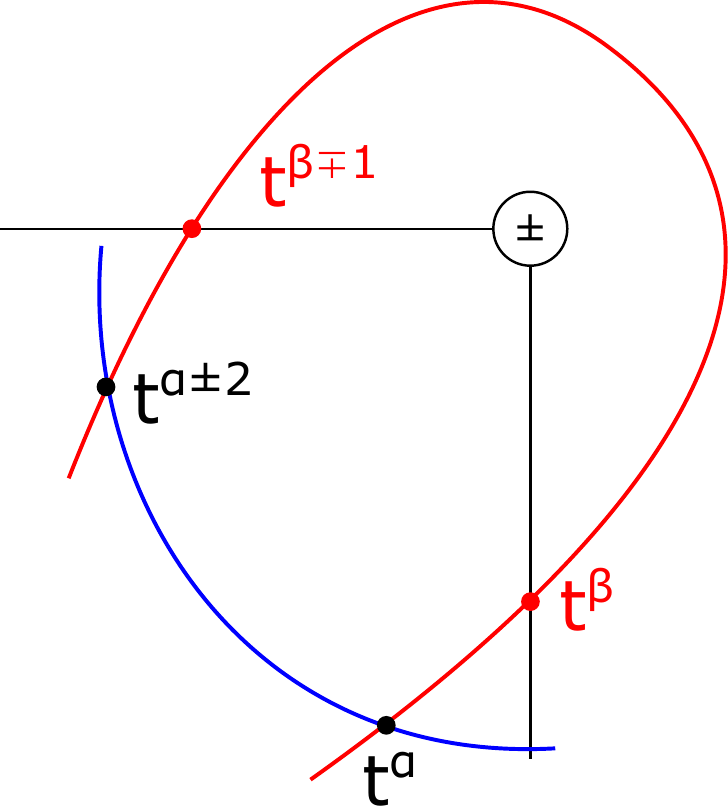}
\caption{A typical ``pair" of intersection points separated by a bigon}
\label{fig:bigon}
\end{figure}
As Theorem \ref{thm:pair} implies, the intersection points between two curves arising from tangles will come in pairs which differ by Alexander grading of 2 and have the same delta grading.  In our examples, they will take the form of pairs lying along parallel arcs, separated by a ``bigon" which covers one puncture, as in Figure \ref{fig:bigon}.  By Remark \ref{rmk:half}, if such a pair has Alexander gradings $\alpha$ and $\alpha \pm 2$, then the corresponding element of $\widehat{HFK}$ will have Alexander grading $\frac{1}{2}(\alpha\pm 1)$.

\subsection*{The immersed curves corresponding to the $(2a,-2b-1)$-pretzel tangle}
\begin{figure}
\centering
\includegraphics[width=.5\textwidth]{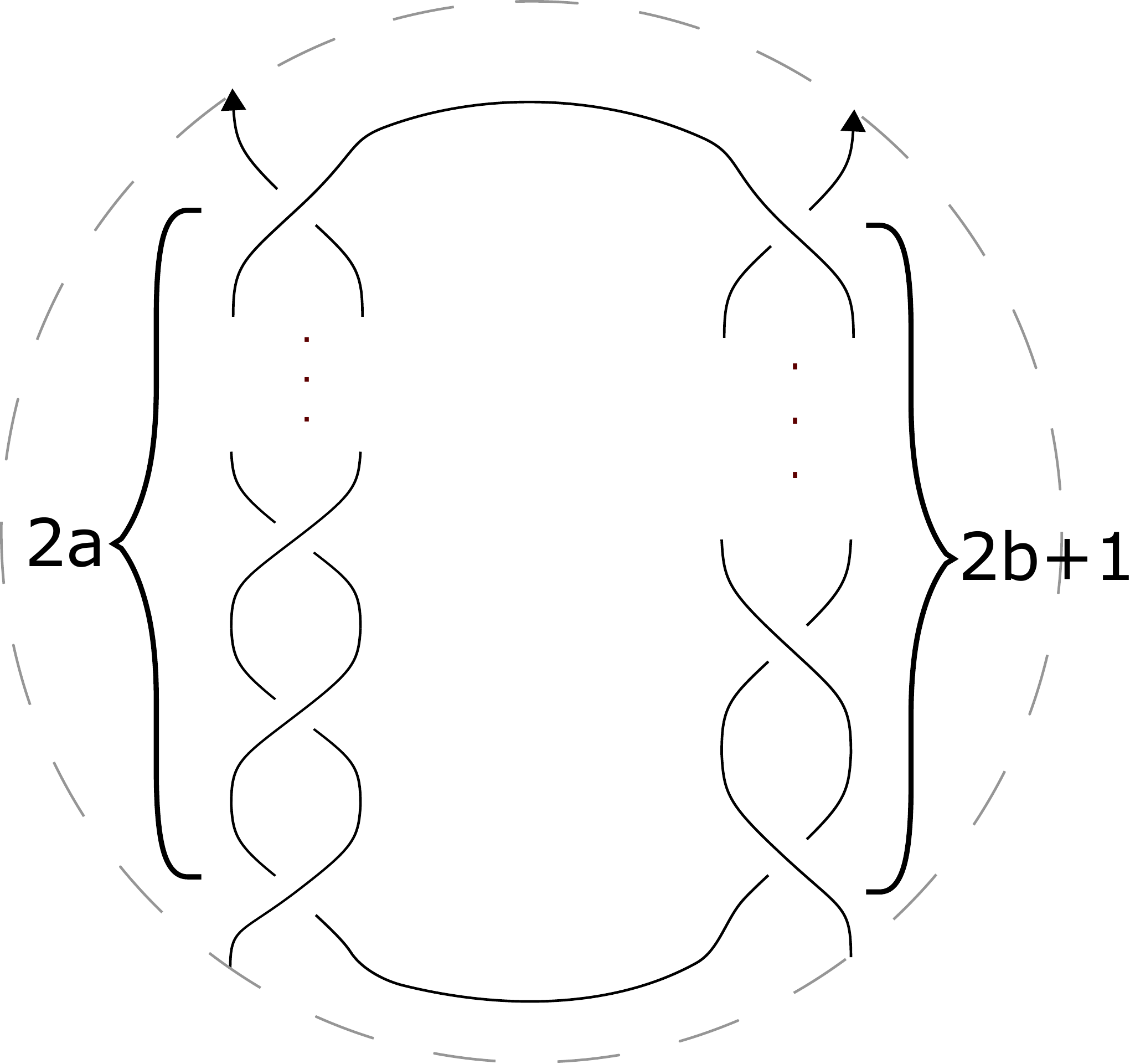}
\caption{The $(2a,-2b-1)$-pretzel tangle}
\label{fig:Ptangle}
\end{figure}
We wish to give a description of the tangle invariant $HFT$ for the $(2a,-2b-1)$-pretzel tangle (pictured in Figure \ref{fig:Ptangle}).  First, we shall recall the two types of standard curve invariants from \cite{Zib}.  These are the \textit{rational} and \textit{special} curves.  Let $\mathbf{T}^2_4$ denote the 4-punctured torus and note that it is a 2-fold covering space of the 4-punctured sphere $\rho: \mathbf{T}^2_4 \rightarrow S^2_4$.  The usual covering of the torus by the plane restricts to a cover $\sigma: \mathbf{R}^2 \setminus \mathbf{Z}^2 \rightarrow \mathbf{T}^2_4$.  Moreover, the maps $\sigma$ and $\sigma$ can be arranged such that the parametrizing square lifts to the ``standard grid" $\mathbf{R}\times \mathbf{Z} \cup \mathbf{Z}\times \mathbf{R}$.

With these conventions, given an extended rational number $\lambda \in \mathbf{Q}\cup\{\infty\}$, the rational curve $r(\lambda)$ is the image under the projection $\rho\circ \sigma$ of a line with slope $\lambda$ (which misses the integer lattice).  It is a fact that if $T$ is a rational tangle with associated rational number $\lambda$ then $HFT(T)$ consists precisely of the curve $r(\lambda)$.  Note also that if $\frac{A}{B}$ is a reduced fraction, then the curve $r\left(\frac{A}{B}\right)$ is an embedded curve which intersects the top and bottom edges of the parametrizing square $|A|$ times each and the left and right edges $|B|$ times each.  See Figures \ref{fig:rat1}--\ref{fig:rat3} for some examples.

The (standard) special curves (also called ``irrational curves" in earlier versions of \cite{Zib}) come in two types: they are denoted $i_k(1,4)$ and $i_k(2,3)$ for each positive integer $k$.  These are projections under $\rho\circ \sigma$ of ``almost" horizontal curves supported near $\mathbf{R}\times \{0\}$ or $\mathbf{R}\times \{1\}$.  These ``almost" horizontal curves repeatedly pass ``above" $2k$ punctures before crossing down and passing ``below" $2k$ punctures and then crossing up again.  In the $(1,4)$ case, the immersed curve ends up supported near a neighborhood of the top edge of the parametrizing square, while in the $(2,3)$ case, it is supported near the bottom edge.  See Figure \ref{fig:irr} for an illustration.

Additionally, given an immersed curve $\Gamma$, we shall encode the (Alexander) grading of this curve by $\Gamma t^mt^M$, where $m$ and $M$ denote the minimal and maximal Alexander grading, respectively, of any intersection point of $\Gamma$ with the parametrizing square.  We shall not explicitly encode the delta grading in such a way; instead, in what follows, we shall always assume curves have delta grading as depicted in Figures \ref{fig:irr}--\ref{fig:rat3}. Using this notation, we may apply  \S 6.3 of \cite{Zib} and write that that the immersed curves corresponding to ${HFT(P(2a,-2b-1))}$ are:

\textbf{Case I: $a\leq b$}\phantomsection\label{c1}
\begin{align*}
& i_1(1,4)t^{-2b-2}t^{-2b+2}\\
& i_1(1,4)t^{-2b}t^{-2b+4}\\
& \quad \vdots \\
& i_{a-1}(1,4)t^{-2b-2}t^{4a-2b-6}\\
& i_{a-1}(1,4)t^{-2b}t^{4a-2b-4}\\
& i_a(1,4)t^{-2b-2}t^{4a-2b-2}\\
& \textstyle r\left(\frac{1}{2a}\right)t^{-2b}t^{4a-2b}\\
& \quad \vdots \\
& \textstyle r\left(\frac{1}{2a}\right)t^{2b-4a}t^{2b}\\
& i_a(2,3)t^{2b-4a+2}t^{2b+2}\\
& i_{a-1}(2,3)t^{2b-4a+4}t^{2b}\\
& i_{a-1}(2,3)t^{2b-4a+6}t^{2b+2}\\
& \quad \vdots \\
& i_1(2,3)t^{2b-4}t^{2b}\\
& i_1(2,3)t^{2b-2}t^{2b+2}
\end{align*}
\textbf{Case II: $a = b + 1$}\phantomsection\label{c2}
\begin{align*}
& i_1(1,4)t^{-2a}t^{-2a+4}\\
& i_1(1,4)t^{-2a+2}t^{-2a+6}\\
& \quad \vdots \\
& i_{a-1}(1,4)t^{-2a}t^{2a-4}\\
& i_{a-1}(1,4)t^{-2a+2}t^{2a-2}\\
& \textstyle r\left(-\frac{1}{2a}\right)t^{-2a}t^{2a}\\
& i_{a-1}(2,3)t^{-2a+2}t^{2a-2}\\
& i_{a-1}(2,3)t^{-2a+4}t^{2a}\\
& \quad \vdots \\
& i_1(2,3)t^{2a-6}t^{2a-2}\\
& i_1(2,3)t^{2a-4}t^{2a}
\end{align*}
\textbf{Case III: $a > b + 1$}\phantomsection\label{c3}
\begin{align*}
& i_1(1,4)t^{-2b-2}t^{-2b+2}\\
& i_1(1,4)t^{-2b}t^{-2b+4}\\
& i_2(1,4)t^{-2b-2}t^{-2b+6}\\
& i_2(1,4)t^{-2b}t^{-2b+8}\\
& \quad \vdots \\
& i_b(1,4)t^{-2b-2}t^{2b-2}\\
& i_b(1,4)t^{-2b}t^{2b}\\
& \textstyle r\left(-\frac{2(a-b)-1}{4b(a-b-1)+2a}\right)t^{-2b-2}t^{2b+2}\\
& i_b(2,3)t^{-2b}t^{2b}\\
& i_b(2,3)t^{-2b+2}t^{2b+2}\\
& i_{b-1}(2,3)t^{-2b+4}t^{2b}\\
& i_{b-1}(2,3)t^{-2b+6}t^{2b+2}\\
& \quad \vdots \\
& i_1(2,3)t^{2b-4}t^{2b}\\
& i_1(2,3)t^{2b-2}t^{2b+2}
\end{align*}

Figures \ref{fig:irr}--\ref{fig:rat3} depict all the different kinds of immersed curves relevant to our calculation.  The curves corresponding to the $(2a,-2b-1)$-pretzel tangle are in blue, while the curves corresponding to the rational tangles $\pm\frac{1}{2c+1}$ are in red.  Note that Figure \ref{fig:rat2} does not depict the most general case of rational curve $r\left(-\frac{A}{B}\right)$.  For our purposes, we are interested, by Case \hyperref[c3]{III}, in the situation where $A=2(a-b)-1$ is odd, $B=4b(a-b-1)+2a$, and $M=2b+2$.

\begin{figure}
\centering
\begin{subfigure}{.45\textwidth}
\centering
\includegraphics[width=\textwidth]{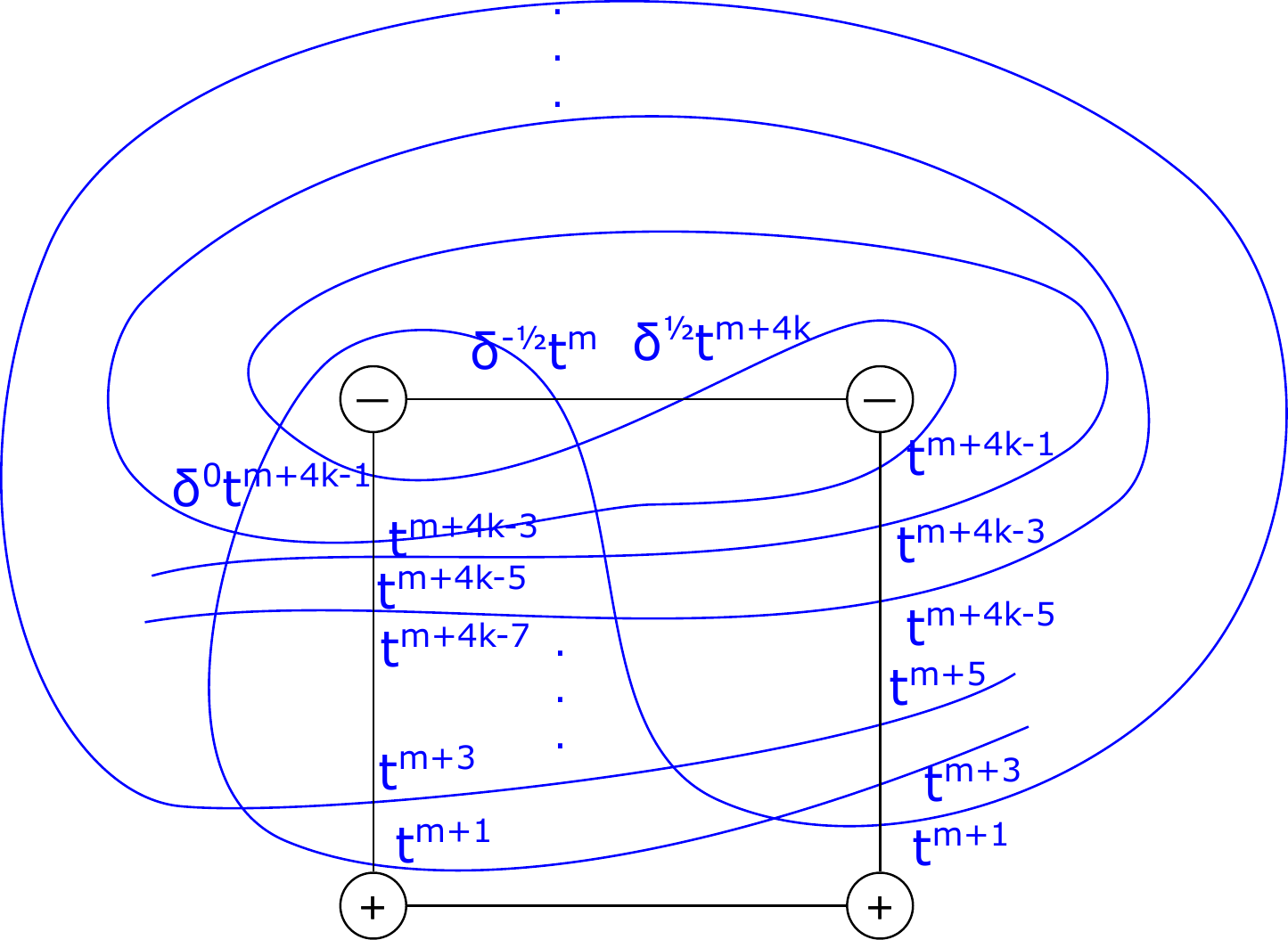}
\caption{The immersed curve associated to $i_k(1,4)t^mt^{m+4k}$}	
\end{subfigure}
\hfill
\begin{subfigure}{.45\textwidth}
\centering
\includegraphics[width=\textwidth]{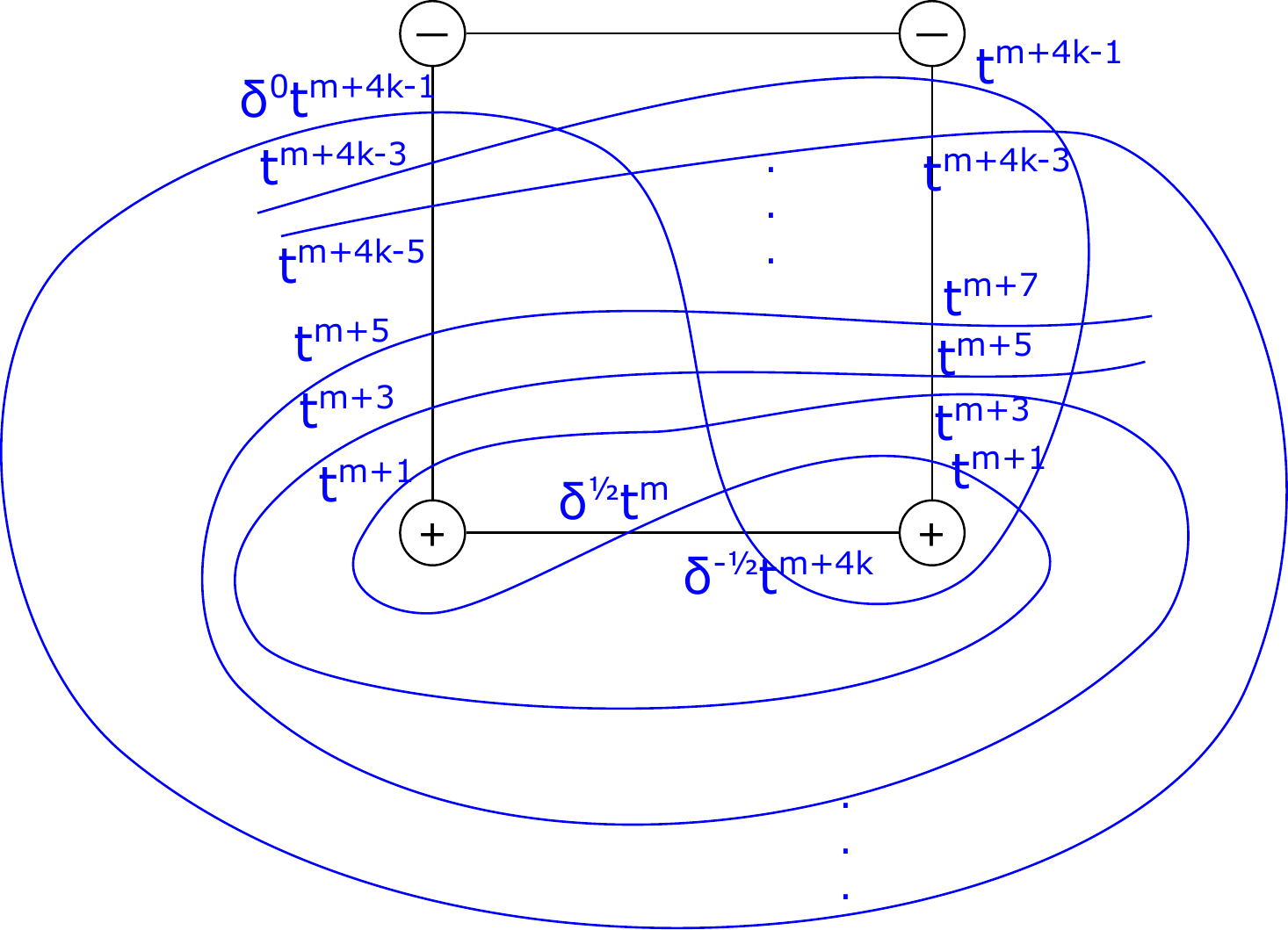}
\caption{The immersed curve associated to $i_k(2,3)t^mt^{m+4k}$}
\end{subfigure}
\caption{Special curves arising in the $HFT$ invariant for the $(2a,-2b-1)$-pretzel tangle}
\label{fig:irr}
\end{figure}

\begin{figure}
\centering
\begin{subfigure}{.45\textwidth}
\centering
\includegraphics[width=\textwidth]{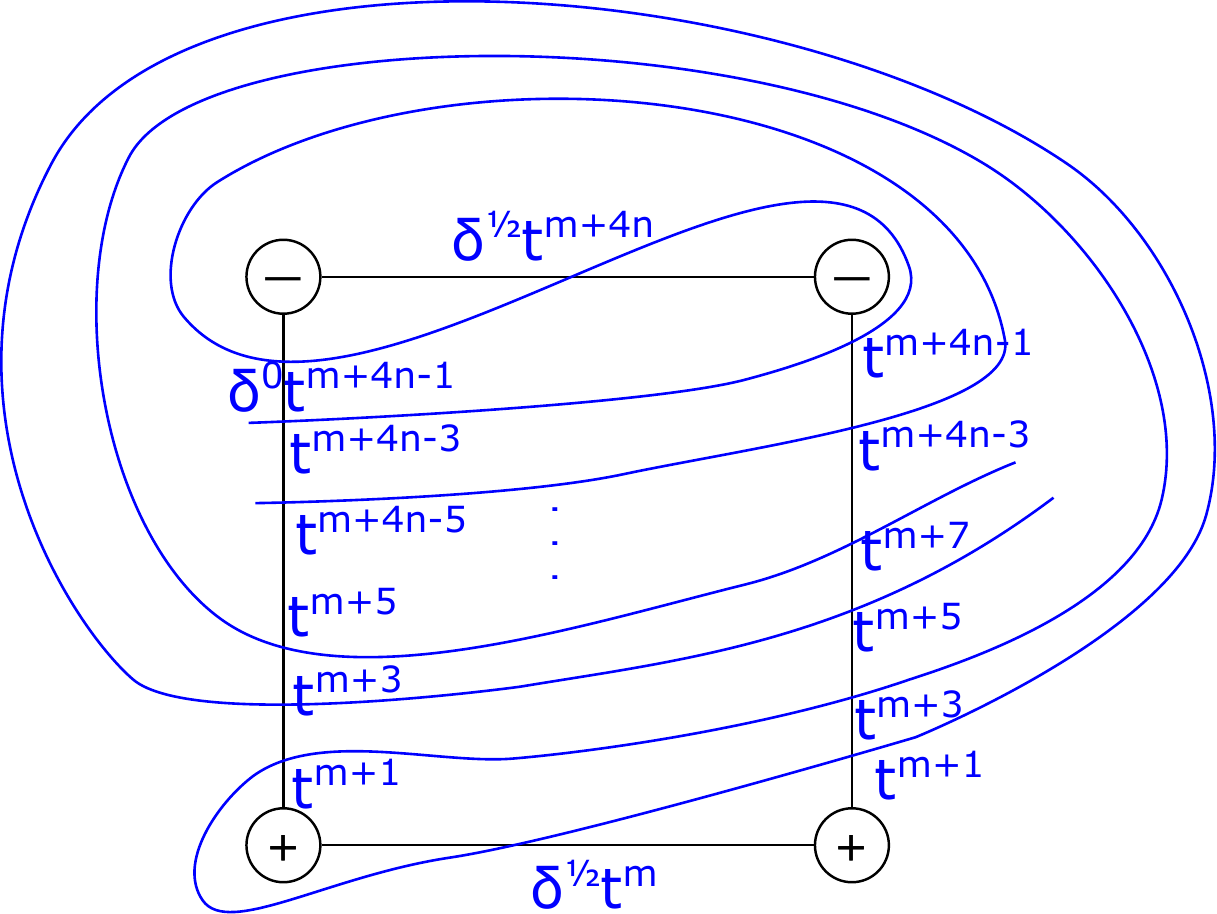}
\caption{The immersed curve associated to $r\left(\frac{1}{2n}\right)t^mt^{m+4n}$}
\end{subfigure}
\hfill
\begin{subfigure}{.45\textwidth}
\centering
\includegraphics[width=\textwidth]{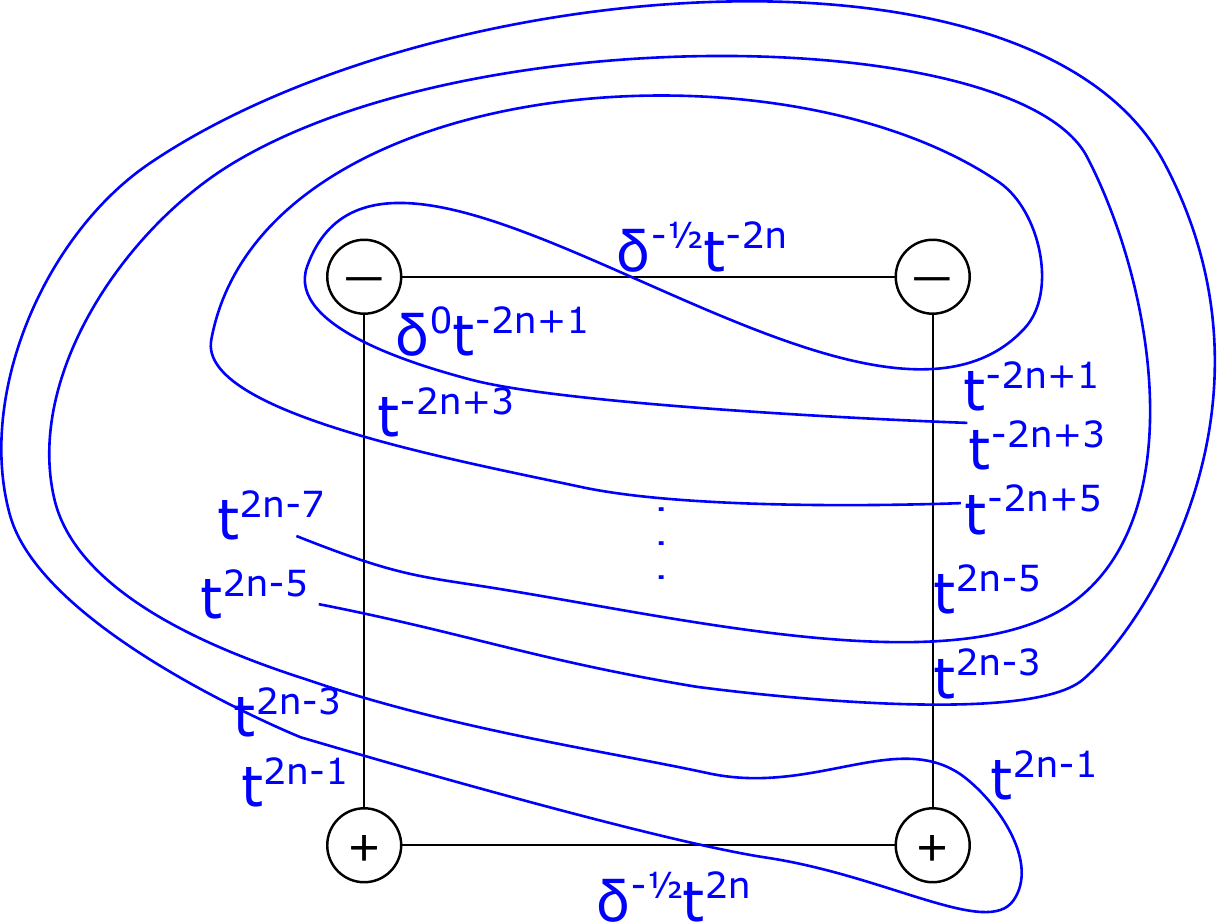}
\caption{The immersed curve associated to $r\left(-\frac{1}{2n}\right)t^{-2n}t^{2n}$}
\end{subfigure}
\caption{Rational curves of slope $\pm \frac{1}{2n}$}
\label{fig:rat1}
\end{figure}

\begin{figure}
\centering
\includegraphics[width=\textwidth]{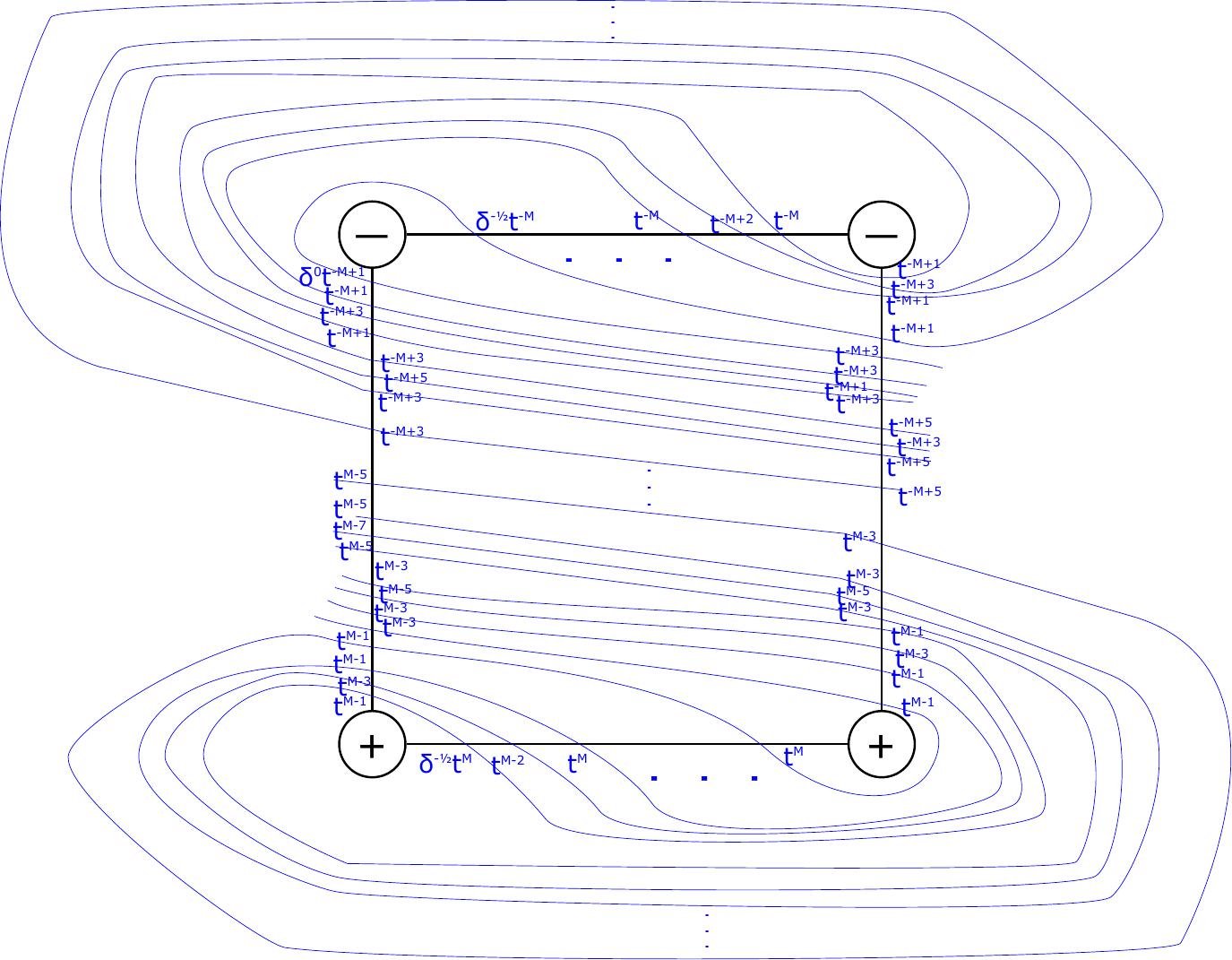}
\caption{The immersed curve associated to $r\left(-\frac{A}{B}\right)t^{-M}t^M$}
\label{fig:rat2}
\end{figure}

\begin{figure}
\centering
\begin{subfigure}{.45\textwidth}
\centering
\includegraphics[width=\textwidth]{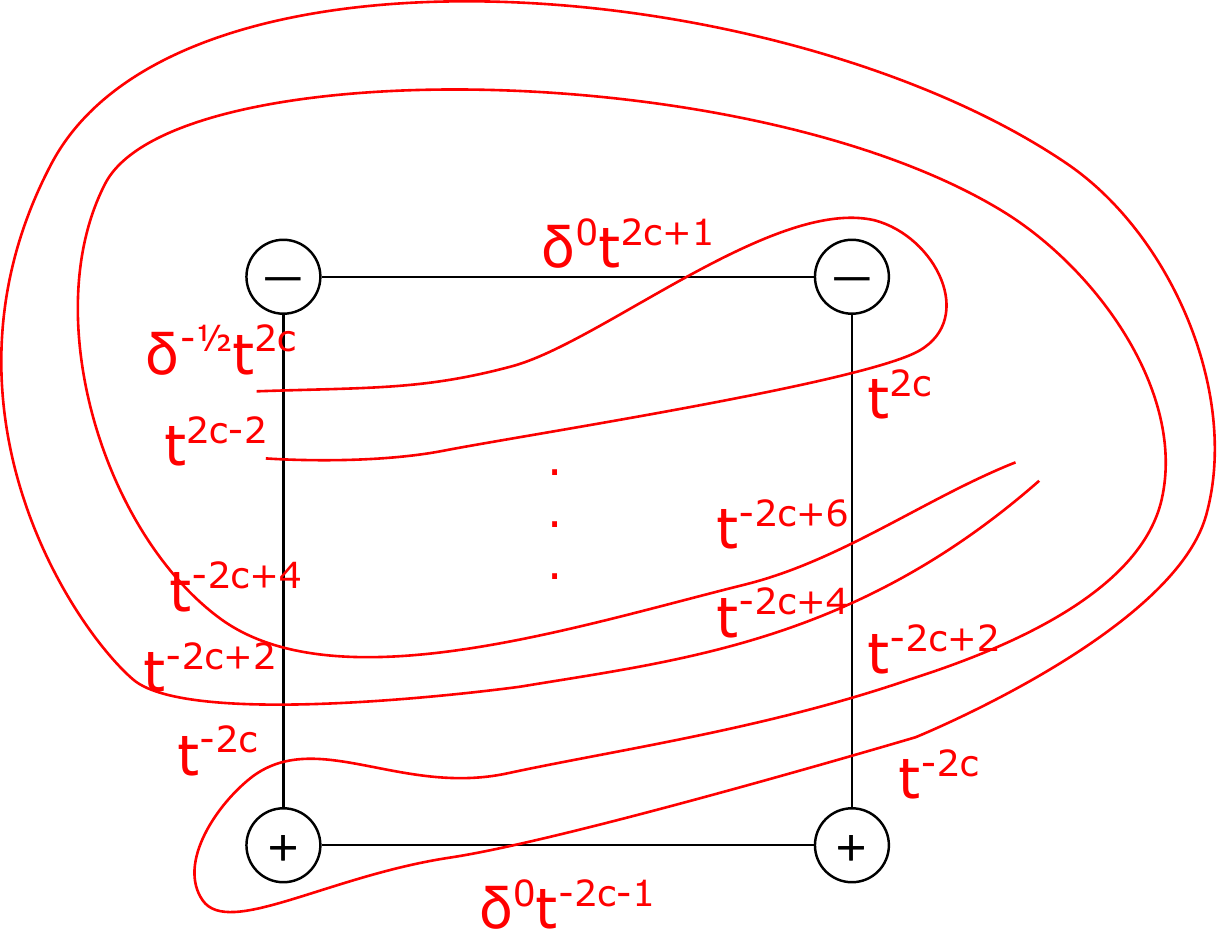}
\caption{The immersed curve associated to $r\left(\frac{1}{2a+1}\right)t^{-2a+1}t^{2a+1}$}
\end{subfigure}
\hfill
\begin{subfigure}{.45\textwidth}
\centering
\includegraphics[width=\textwidth]{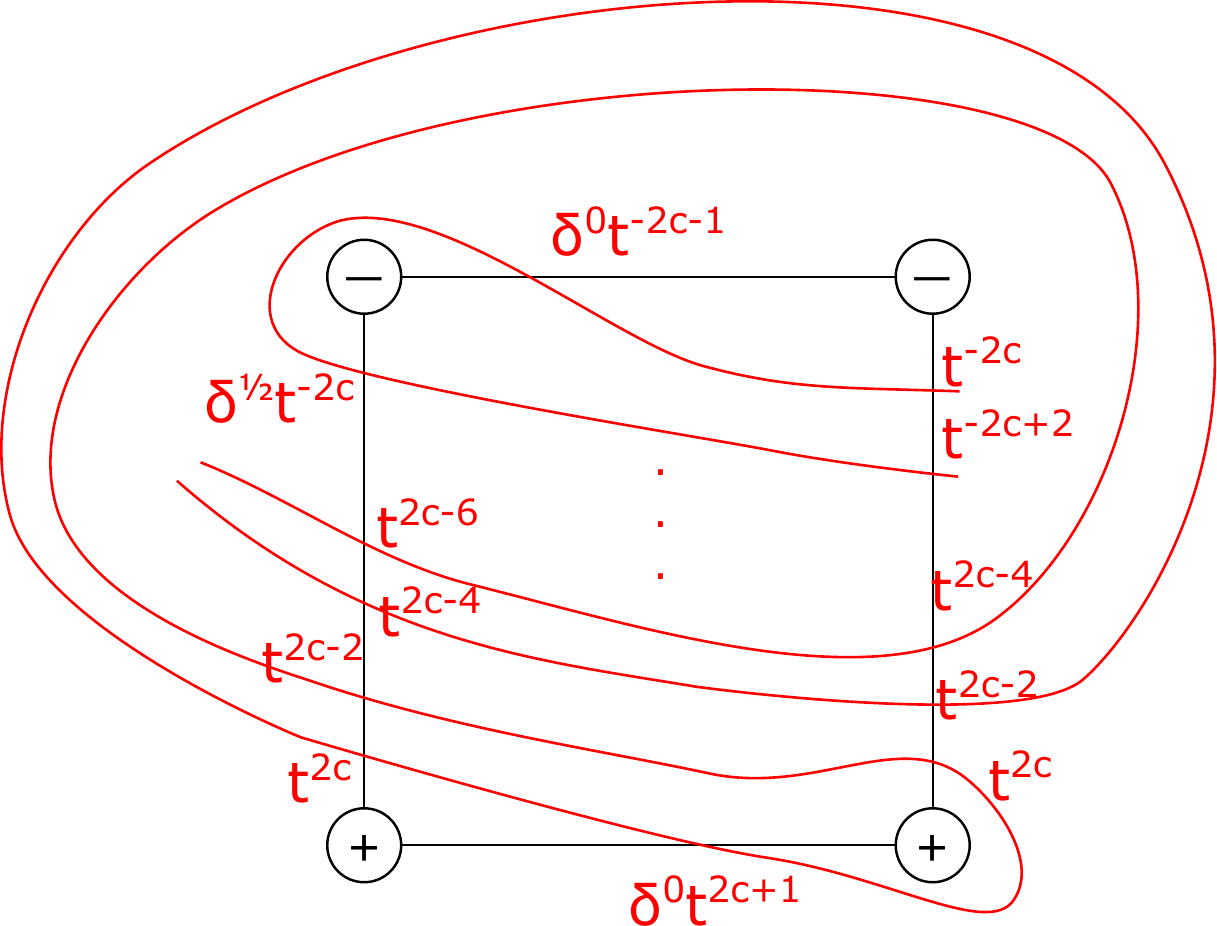}
\caption{The immersed curve associated to $r\left(-\frac{1}{2a+1}\right)t^{-2a-1}t^{2a+1}$}
\end{subfigure}
\caption{The curves corresponding to the $\pm \frac{1}{2a+1}$ rational tangles}
\label{fig:rat3}
\end{figure}

\begin{rmk}[A note on gradings]
The Alexander and delta gradings on immersed curves are \textit{a priori} relative.  For the Alexander gradings, symmetrizing them will result in them being symmetrized when computing $\widehat{HFK}$.  However, our delta grading will remain relative throughout this work; the absolute grading may be recovered by considering the Kaufmann state interpretation (see \cite{OSzBord}) as done by Eftekhary.  Furthermore, in the interest of space, the delta gradings of many of the intersection points in Figures \ref{fig:irr} -- \ref{fig:rat3} have been suppressed.  All intersection points on the left and right edges of the parametrizing square have the same delta grading, and so only the top-left intersection is labelled explicitly once.  Similarly if only one intersection point on the top or bottom edge has an explicit delta grading, it should be interpreted that all other intersection points along that edge have the same grading.
\end{rmk}

\section{The case $K=P(2a,-2b-1,-2c-1)$}
\begin{figure}
\centering
\includegraphics[width=.75\textwidth]{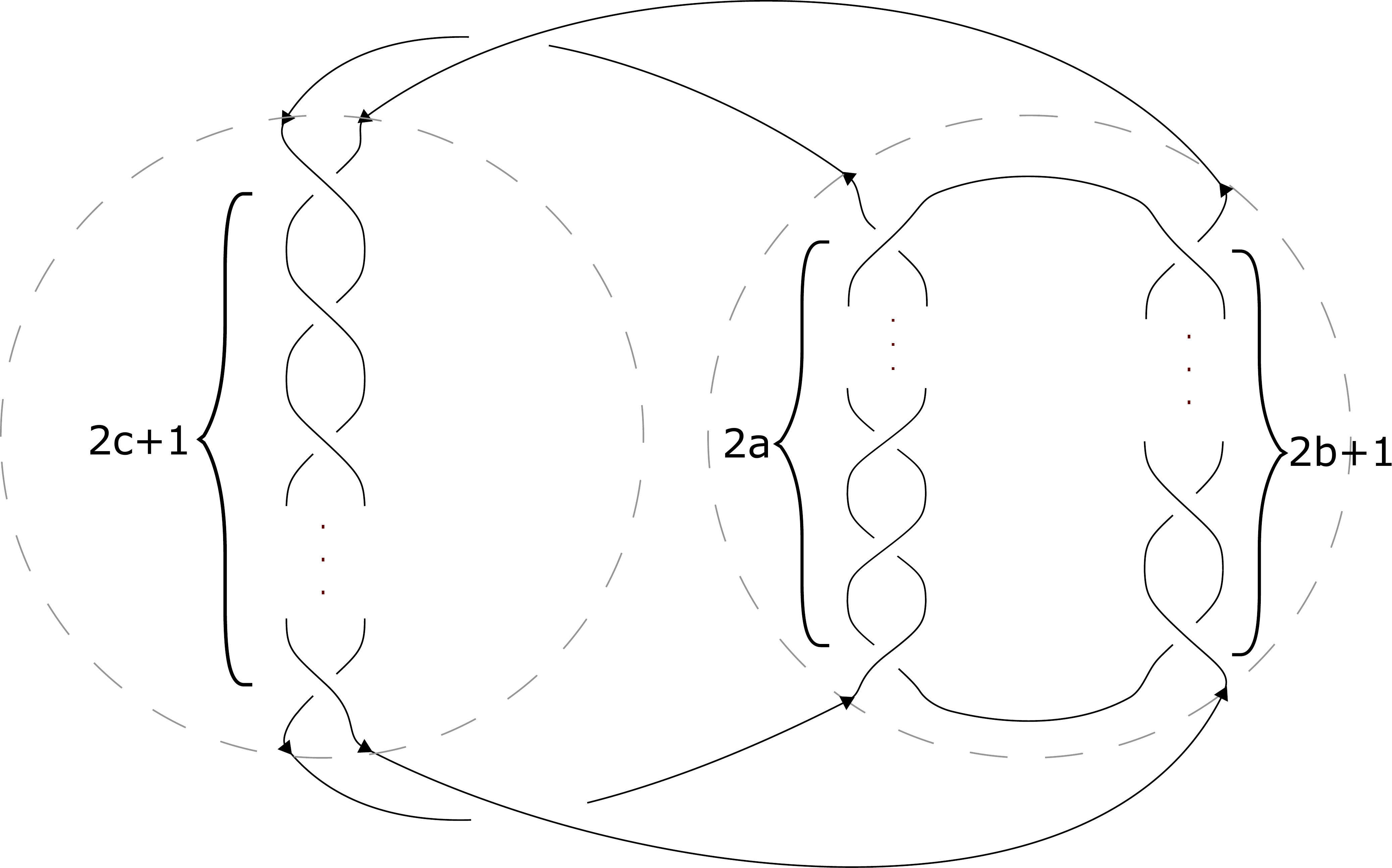}
\caption{The knot $P(2a,-2b-1,-2c-1)$ realized as the pairing of two tangles}
\label{fig:pair_N}
\end{figure}
In this section, we devote our attention to the case of the pretzel knot ${P(2a,-2b-1,-2c-1)}$.  As this corresponds to pairing the $-\frac{1}{2c+1}$ rational tangle with the ${(2a,-2b-1)}$-pretzel tangle (see Figure \ref{fig:pair_N}), we need to compute the pairing of the curves from the previous section with the rational curve $r\left(\frac{1}{2c+1}\right)t^{-2c-1}t^{2c+1}$.  Looking through Cases \hyperref[c1]{I}--\hyperref[c3]{III} above, this requires pairing with curves of the form:
\begin{itemize}
\item $i_k(1,4)t^mt^M$
\item $i_k(2,3)t^mt^M$
\item $r\left(-\frac{1}{2n}\right)t^{-2n}t^{2n}$
\item $r\left(\frac{1}{2n}\right)t^{m}t^{M}$
\item $r\left(-\frac{A}{B}\right)t^{-M}t^{M}$
\end{itemize}
Here $A=2(a-b)-1$ and $B=4b(a-b-1)+2a$.  Note also that $m$ and $M$ are not truly independent of each other, but considering them separately will help make the symmetry in the Alexander gradings more apparent.
\begin{figure}[h]
\includegraphics[width=\textwidth]{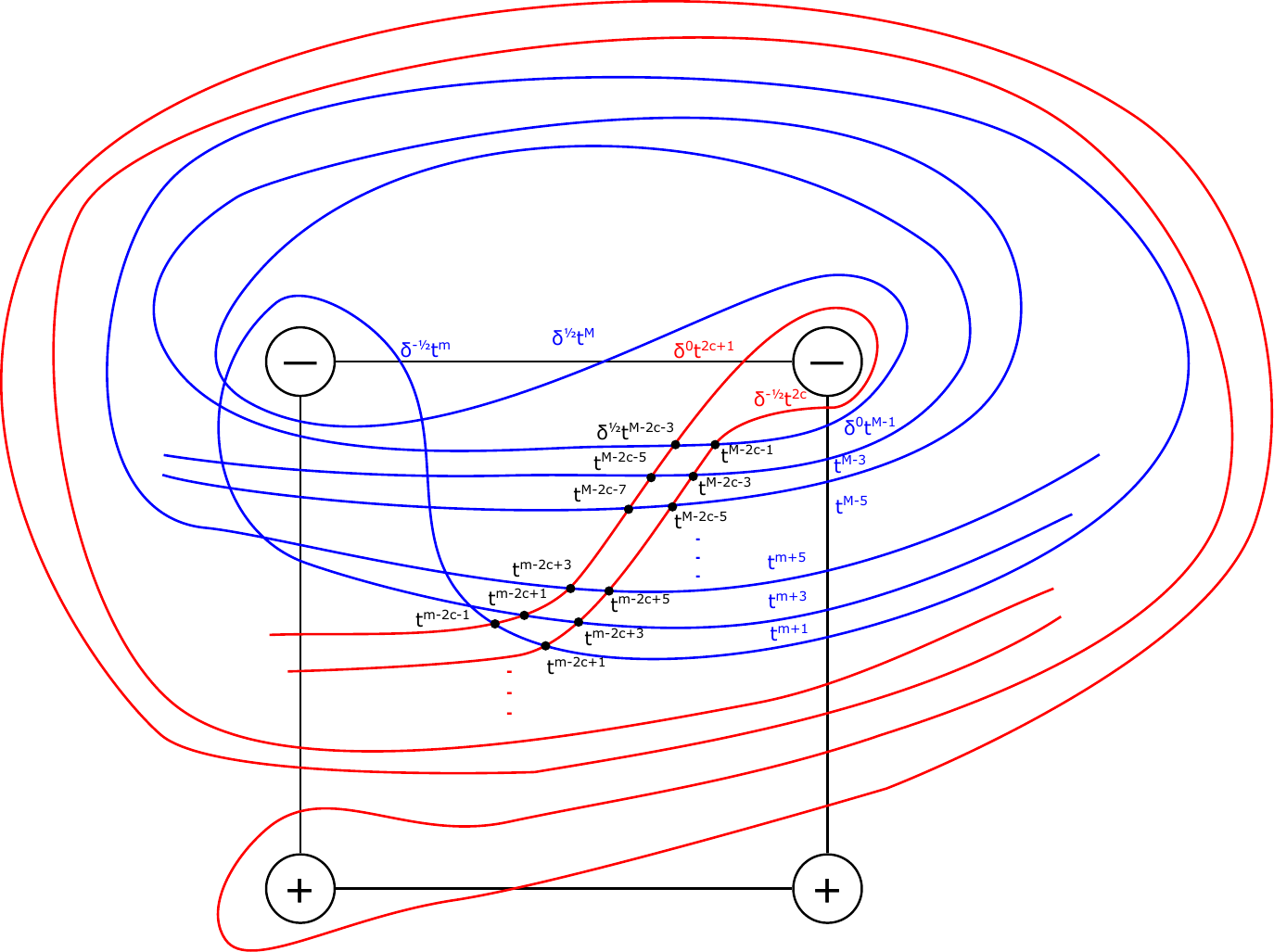}
\caption{The minimal intersection between the curves $i_k(1,4)t^mt^M$ and $r\left(\frac{1}{2c+1}\right)t^{-2c-1}t^{2c+1}$}
\label{fig:int_i14p}
\end{figure}

We first consider the pairing $HF\left(r\left(\frac{1}{2c+1}\right)t^{-2c-1}t^{2c+1},i_k(1,4)t^mt^M\right)$. As Figure \ref{fig:int_i14p} indicates, the curves may be isotoped so that all the intersections occur in the upper part of the diagram, with the red curve crossing through all the blue arcs, followed by ``rotations" which do not meet the blue curve at all. Using \eqref{eq:int_D}, one sees that the intersection points all have delta grading $\frac{1}{2}$.  Moreover, by \eqref{eq:int_A}, they come in pairs such that in each pair, the Alexander gradings differ by 2, and consecutive pairs differ by an overall Alexander grading shift by 2.  In particular, the pairs of gradings are:
\begin{align*}
(M-2c-3,M-2c-1), (M-2c-5,M-2c-3), (M-2c-7,M-2c-5),\dots \\
\dots,(m-2c+3,m-2c+5), (m-2c+1,m-2c+3), (m-2c-1,m-2c+1)
\end{align*}
In light of Theorem \ref{thm:pair}, we deduce that each such pair of generators corresponds to a single generator of $\widehat{HFK}(K)$ with Alexander grading equal to (by Remark \ref{rmk:half}, half of) the average of each pair.  This means that the Alexander gradings in the Knot Floer Homology are:
\begin{align*}
\frac{M}{2}-c-1, \frac{M}{2}-c-2, \frac{M}{2}-c-3,\dots,\frac{m}{2}-c+2, \frac{m}{2}-c+1, \frac{m}{2}-c
\end{align*}
For convenience, we denote the ``reduced" (that is, with each appropriate pair replaced by the corresponding $\widehat{HFK}$ generator and Alexander grading scaled appropriately) Floer homology of two immersed curves $\Gamma_1$ and $\Gamma_2$ by $\Gamma_1 \boxtimes \Gamma_2$.  Hence, we have shown:
\begin{equation}\label{eq:I14_1}
\begin{aligned}
r\left(\frac{1}{2c+1}\right)t^{-(2c+1)}t^{2c+1} \boxtimes i_k(1,4)t^mt^M \\
=\delta^{1/2}\{t^{m/2-c},t^{m/2-c+1},\dots,t^{M/2-c-1}\}
\end{aligned}
\end{equation}
Since all generators have the same delta grading, we save space by only writing $\delta^{1/2}$ once and putting it outside the curly brackets.

\begin{figure}[!h]
\includegraphics[width=\textwidth]{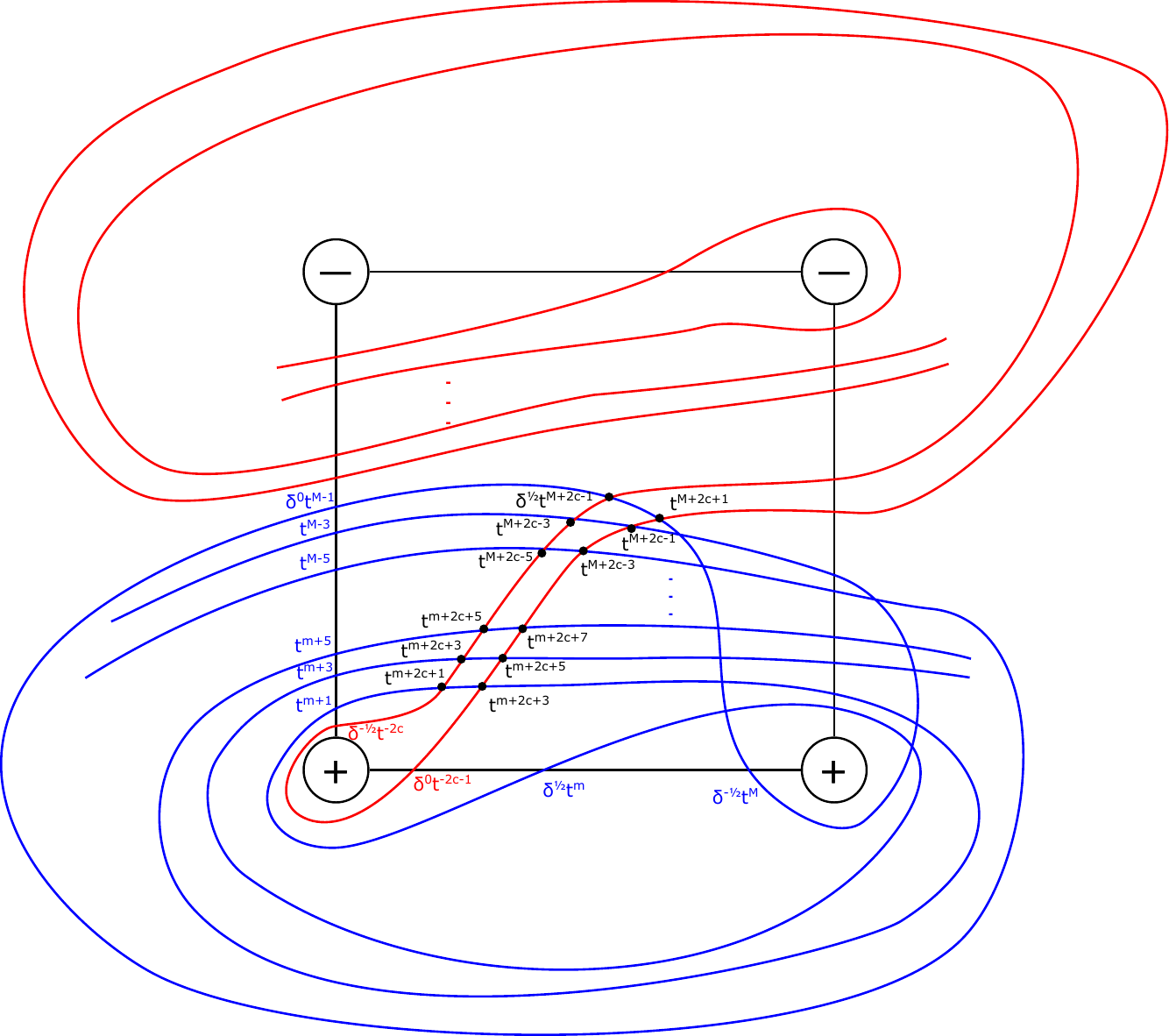}
\caption{The minimal intersection between the curves $i_k(2,3)t^mt^M$ and $r\left(\frac{1}{2c+1}\right)t^{-2c-1}t^{2c+1}$}
\label{fig:int_i23p}
\end{figure}
The pairing with $i_k(2,3)t^mt^M$ is similar to the previous pairing.  Once again, observing Figure \ref{fig:int_i23p} we see all generators with delta grading $\frac{1}{2}$ and the Alexander gradings arranged in consecutive pairs, each pair shifted by 2 from the previous.  In particular, the Alexander gradings are:
\begin{align*}
(M+2c+1,M-2c-1), (M+2c-1,M-2c-3), (M+2c-3,M-2c-5),\dots\\
\dots,(m+2c+7,m-2c+5), (m+2c+5,m-2c+3), (m+2c+3,m-2c+1)
\end{align*}
Hence, applying the ``reduction", we see that:
\begin{equation}\label{eq:I23_1}
\begin{aligned}
r\left(\frac{1}{2c+1}\right)t^{-(2c+1)}t^{2c+1} \boxtimes i_k(2,3)t^mt^M \\
=\delta^{1/2}\{t^{m/2+c+1},t^{m/2+c+2},\dots,t^{M/2+c}\}
\end{aligned}
\end{equation}

\begin{figure}[h]
\includegraphics[width=\textwidth]{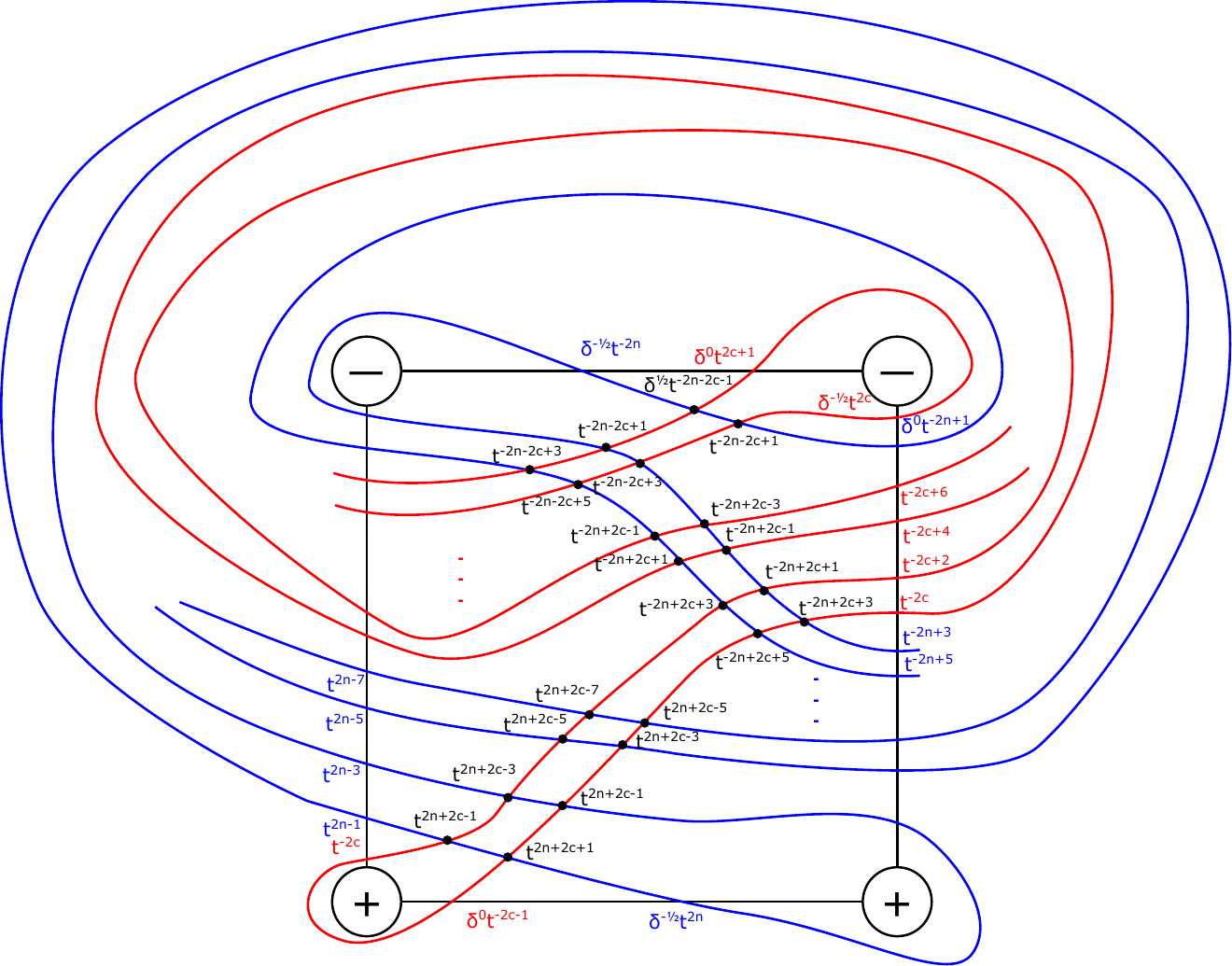}
\caption{The minimal intersection between the curves $r\left(-\frac{1}{2n}\right)t^{-2n}t^{2n}$ and $r\left(\frac{1}{2c+1}\right)t^{-2c-1}t^{2c+1}$}
\label{fig:int_rnp}
\end{figure}
We now consider the pairing with the rational curve $r\left(-\frac{1}{2n}\right)t^{-2n}t^{2n}$.  By Figure \ref{fig:int_rnp}, one sees once again all delta gradings $\frac{1}{2}$ as well as a pattern of consecutive pairs of Alexander gradings:
\begin{align*}
(-2n-2c-1,-2n-2c+1),(-2n-2c+1,-2n-2c+3),(-2n-2c+3,-2n-2c+5),\dots\\
\dots,(2n+2c-3,2n+2c-5),(2n+2c-1,2n+2c-3),(2n+2c+1,2n+2c-1)
\end{align*}
Hence we have:
\begin{equation}\label{eq:RN_1}
\begin{aligned}
r\left(\frac{1}{2c+1}\right)t^{-(2c+1)}t^{2c+1} \boxtimes r\left(-\frac{1}{2n}\right)t^{-2n}t^{2n} \\
=\delta^{1/2}\{t^{-n-c},t^{-n-c+1},\dots,t^{n+c}\}
\end{aligned}
\end{equation}
Alternatively, one can compute this by noticing that such a pairing corresponds to pairing the rational tangles $-\frac{1}{2c+1}$ and $-\frac{1}{2n}$.  This gives the $(2,-2(c+n)-1)$-torus knot, which is known to have Knot Floer homology concentrated in a single delta grading and a single generator in each Alexander grading from $-c-n$ up to $c+n$.

Moreover, we record the following useful result regarding pairing rational curves:
\begin{lem}\label{lem:rat_det}
Let $\frac{r_1}{s_1}$ and $\frac{r_2}{s_2}$ be two rational numbers (in simplest terms).  Then $HF\left(r\left(\frac{r_1}{s_1}\right),r\left(\frac{r_2}{s_2}\right)\right)$ has $\Delta = \left|\det\left(\begin{matrix}
r_1 & r_2\\
s_1 & s_2
\end{matrix}\right)\right|$ pairs of generators.  Equivalently, $r\left(\frac{r_1}{s_1}\right) \boxtimes r\left(\frac{r_2}{s_2}\right)$ has $\Delta$ generators.
\end{lem}
\begin{proof}
Recalling the double cover $\rho:T^2_4 \rightarrow S^2_4$ and the definition of rational curves, we know that each rational curve $r\left(\frac{r_i}{s_i}\right)$ lifts to two curves $\gamma^1_i$ and $\gamma^2_i$ on the 4-punctured torus, each with slope $\frac{r_i}{s_i}$.  Now the (minimal) intersection number of $\gamma^j_1$ and $\gamma^k_2$ is $\Delta$, whereas the curves $\gamma^1_i$ and $\gamma^2_i$ are disjoint, being separate lifts of the same curve in $S^2_4$.  Hence, there are $4\Delta$ total points of intersection among all the $\gamma^j_i$ curves.  Since $\rho$ is a double cover, this means that there are $2\Delta$ intersection points between $r\left(\frac{r_1}{s_1}\right)$ and $r\left(\frac{r_2}{s_2}\right)$ (in minimal-intersection position); that is, $\Delta$ pairs.
\end{proof}
The above Lemma is consistent with our calculation for this case, since $\det\left(\begin{matrix}
1 & -1\\
2c+1 & 2n
\end{matrix}\right) = 2(n+c)+1$.

Another useful fact about pairing two rational curves is the following:
\begin{lem}\label{lem:rat_delta}
If $\lambda_1$ and $\lambda_2$ are (extended) rational numbers. Then ${HF(r(\lambda_1),r(\lambda_2))}$ is supported in a single delta-grading.
\end{lem}
\begin{proof}
Since this is the pairing of two rational curves, it may, as we have seen, be interpreted (up to grading shifts) as the Knot Floer homology of the pairing of two rational tangles (actually, this will result in a link in general, so we need to consider Link Floer homology; there is an analogue of Theorem \ref{thm:pair} for links as well; see Theorem 5.9 of \cite{Zib}).  This will be a 2-bridge link, which is known to have Link Floer homology supported in a single delta-grading.
\end{proof}

\begin{figure}[h]
\centering
\begin{subfigure}{\textwidth}
\centering
\includegraphics[width=\textwidth]{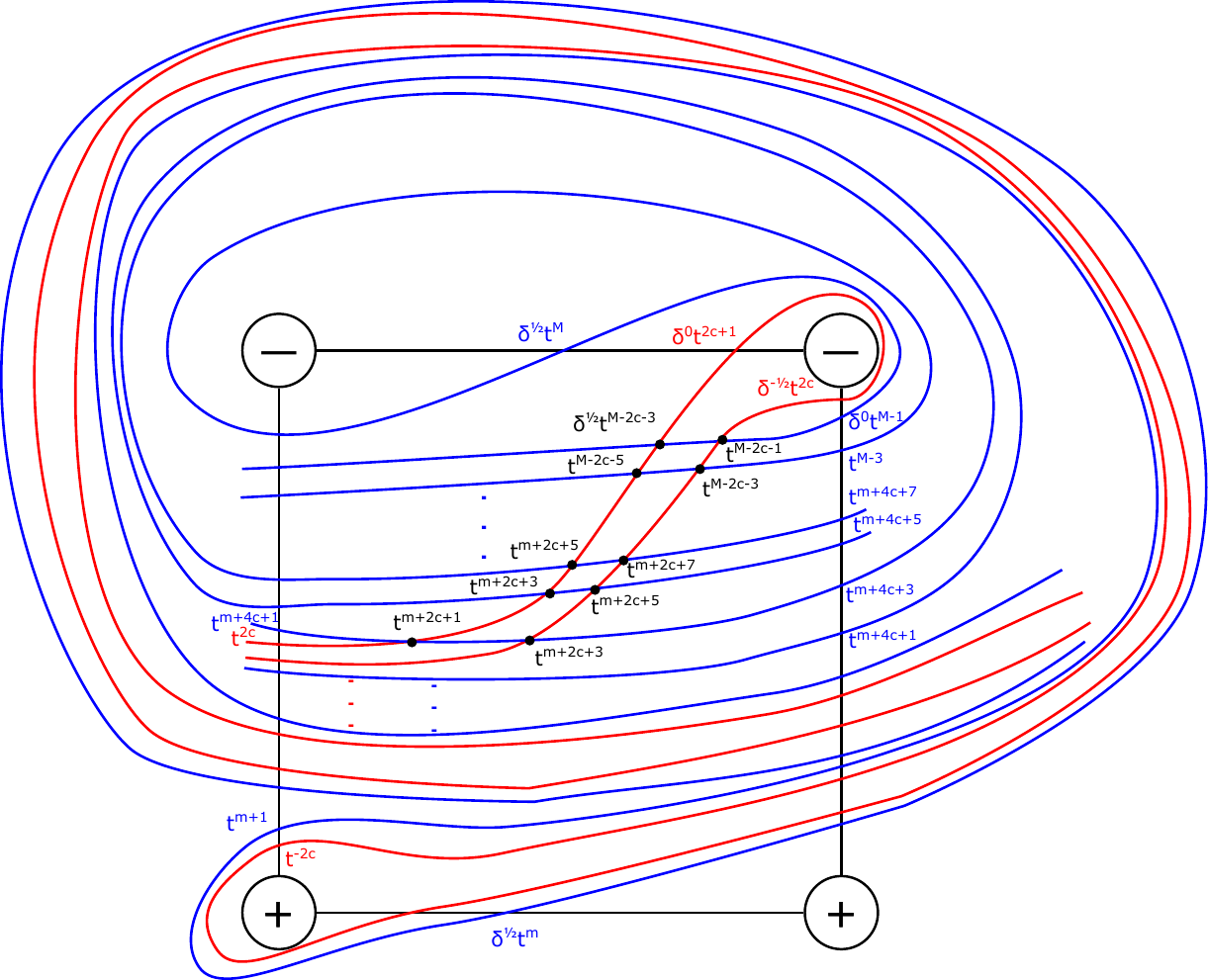}
\caption{The case where $n>c$}
\label{fig:int_rpp1}
\end{subfigure}
\caption{The minimal intersection between the curves $r\left(\frac{1}{2n}\right)t^{m}t^{M}$ and $r\left(\frac{1}{2c+1}\right)t^{-2c-1}t^{2c+1}$; the two cases}
\label{fig:int_rpp}
\end{figure}
\begin{figure}\ContinuedFloat
\begin{subfigure}{\textwidth}
\includegraphics[width=\textwidth]{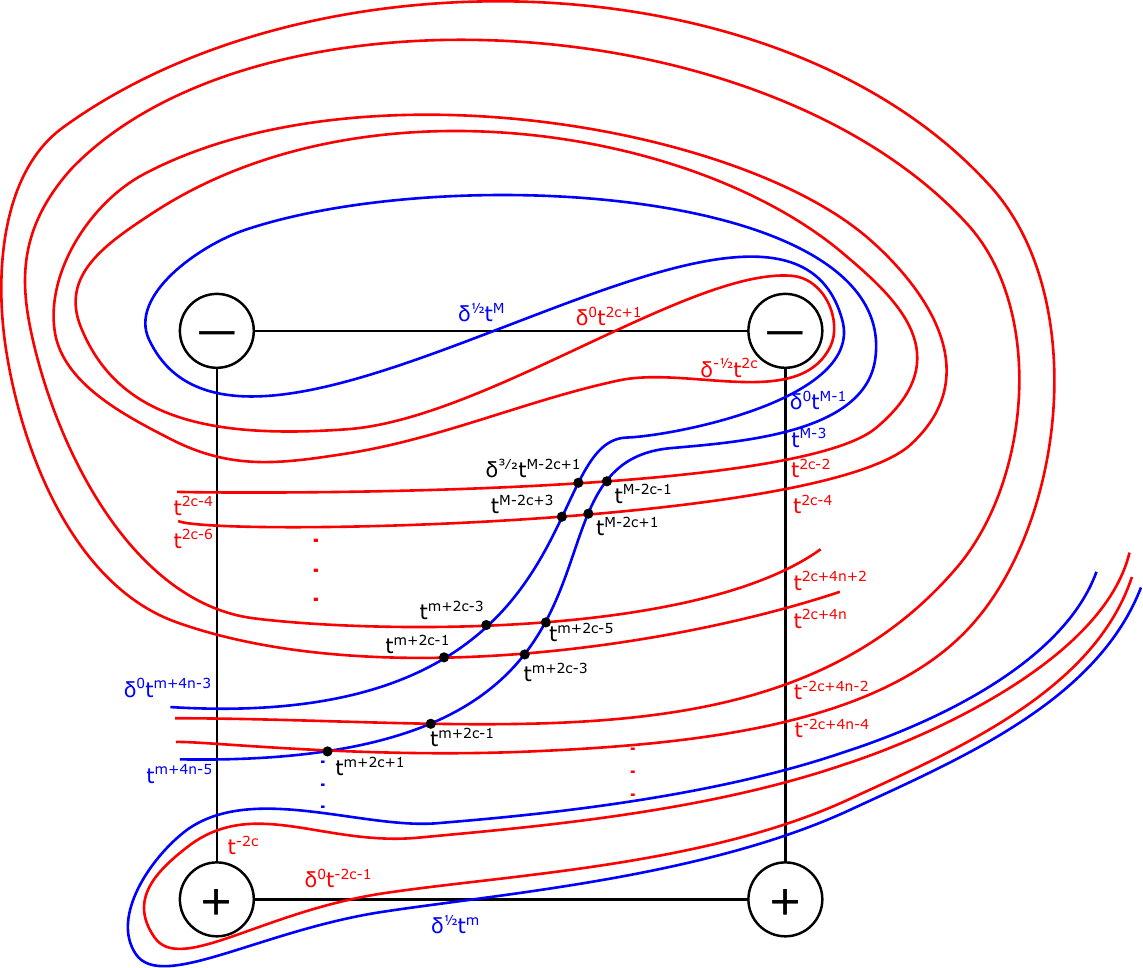}
\caption{The case where $n\leq c$}
\label{fig:int_rpp2}
\end{subfigure}
\end{figure}
Now let us consider the pairing with the curve $r\left(\frac{1}{2n}\right)t^{m}t^{M}$.  In this case, the rational curves both have positive slope, and as shown in Figure \ref{fig:int_rpp}, there are two different subcases to consider.  In one case, if $\frac{1}{2n}<\frac{1}{2c+1}$ (or, equivalently, $n>c$), we can put the curves in the following configuration: the curve $r\left(\frac{1}{2c+1}\right)$ (the red curve in the figure) follows the curve $r\left(\frac{1}{2n}\right)$ (the blue curve) as shown, until all its ``rotations" have finished, and then it crosses the remaining blue arcs to reach the top-right corner. The fact that the red curve finishes its $c$ ``rotations" no later than the blue finishes its $n-1$ ``rotations" is guaranteed by the condition that $n>c$; also notice that each ``rotation" results in an an increase by 4 in the Alexander grading of intersections between the curves and the parametrizing square.  Hence, we see from Figure \ref{fig:int_rpp1} that the intersection points will all have delta-grading $\frac{1}{2}$ (in light of Lemma \ref{lem:rat_delta}, this only needs to be checked for one point) and will have consecutive pairs of Alexander gradings:
\begin{align*}
(m+2c+1,m+2c+3),(m+2c+3,m+2c+5),(m+2c+5,m+2c+7),\dots\\
\dots,(M-2c-5,M-2c-3),(M-2c-3,M-2c-1)
\end{align*}

In the second case, we see, as in Figure \ref{fig:int_rpp2}, that we can isotope the curves such that the blue curve this time follows the red one until its ``rotations" are complete, and then crosses over the remaining red arcs before looping around to the top-left puncture.  From the figure, one sees that all the intersection points have delta-grading $\frac{3}{2}$ and have Alexander gradings once again arranged in consecutive pairs:
\begin{align*}
(m+2c+1,m+2c-1),(m+2c-1,m+2c-3),(m+2c-3,m+2c-5),\dots\\
\dots,(M-2c+3,M-2c+1),(M-2c+1,M-2c-1)
\end{align*}

Therefore, applying the reduction to the results of the two subcases, we have:
\begin{equation}\label{eq:RP_1}
\begin{aligned}
r\left(\frac{1}{2c+1}\right)t^{-(2c+1)}t^{2c+1} \boxtimes r\left(\frac{1}{2n}\right)t^{m}t^{M} \\
=\begin{cases}
\delta^{1/2}\{t^{m/2+c+1},t^{m/2+c+2},\dots,t^{M/2-c-1}\} & \textrm{ if }n>c \\
\delta^{3/2}\{t^{M/2-c},t^{M/2-c+1},\dots,t^{m/2+c}\} & \textrm{ if }n\leq c
\end{cases}
\end{aligned}
\end{equation}
Once again, this can be seen as corresponding to the Knot Floer homology of the $(2,2n-2c-1)$-torus knot (up to an overall shift in Alexander grading).

\begin{figure}[p]
\includegraphics[width=\textwidth]{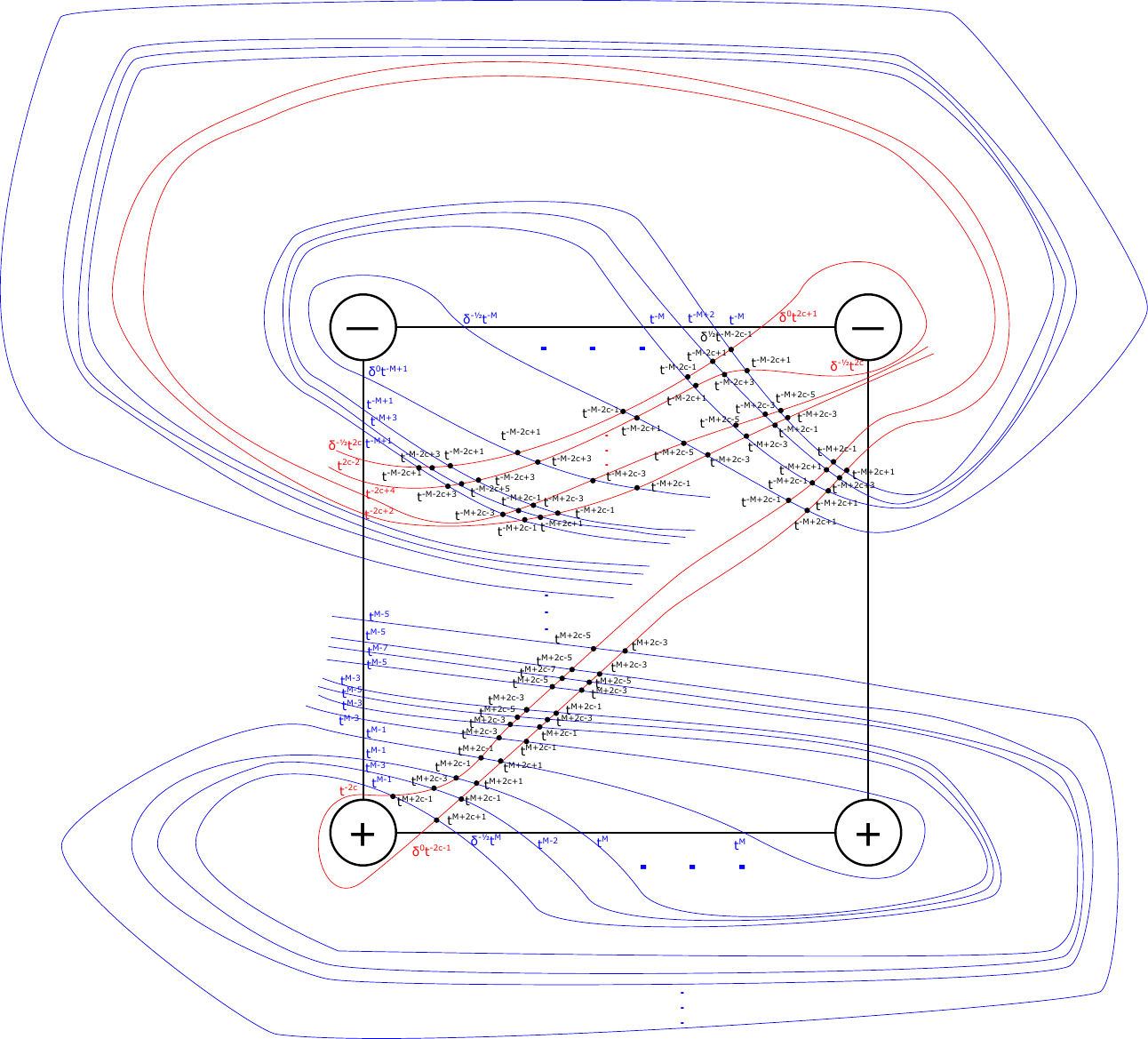}
\caption{The minimal intersection between the curves $r\left(-\frac{A}{B}\right)t^{-M}t^{M}$ and $r\left(\frac{1}{2c+1}\right)t^{-2c-1}t^{2c+1}$}
\label{fig:int_rgp}
\end{figure}
We now turn our attention to the pairing with the more general rational curve $r\left(-\frac{A}{B}\right)t^{-M}t^M$, where, by Case \hyperref[c3]{III}, we have $A=2(a-b)-1$, $B=4b(a-b-1)+2a=A(2b+1)+1$, and $M=2b+2$.  The curves can be isotoped into the following configuration: the red curve $r\left(\frac{1}{2c+1}\right)$, starting at the top-right puncture, it completes all of its ``rotations" by repeatedly crossing only the top $2A$ blue arcs in ``inside" of the parametrizing square; Then, it crosses the rest of the arcs to reach the bottom-left puncture (see Figure \ref{fig:int_rgp}).  From this, we see that all intersection points have delta-grading $\frac{1}{2}$ (again, by Lemma \ref{lem:rat_delta}, it suffices to check this for one intersection point).  Next we observe whenever the red curve crosses a ``group" of $A$ arcs, the pairs of Alexander gradings will alternate in the following fashion:
\begin{align*}
(\alpha-1,\alpha+1),(\alpha+1,\alpha+3),(\alpha-1,\alpha+1),(\alpha+1,\alpha+3),(\alpha-1,\alpha+1)...\\
...(\alpha-1,\alpha+1),(\alpha+1,\alpha+3),(\alpha-1,\alpha+1)
\end{align*}
Moreover, between two consecutive such ``groups", there is an overall grading shift of 2.  By Lemma \ref{lem:rat_det}, there will be $L = B+A(2c+1)$ total pairs.  Applying the reduction gives:
\begin{equation}\label{eq:RG_1}
\begin{aligned}
& r\left(\frac{1}{2c+1}\right)t^{-(2c+1)}t^{2c+1} \boxtimes r\left(-\frac{A}{B}\right)t^{-M}t^{M} \\
& =\delta^{1/2} \\
& \left.
\begin{aligned}
& \left.
\begin{aligned}
& t^{-M/2-c}\\
& t^{-M/2-c+1}\\
& t^{-M/2-c}\\
& \vdots\\
& t^{-M/2-c}
\end{aligned}\right\rbrace A \\
&\left.
\begin{aligned}
& t^{-M/2-c+1}\\
& t^{-M/2-c+2}\\
& t^{-M/2-c+1}\\
& \vdots\\
& t^{-M/2-c+1}
\end{aligned}\right\rbrace A \\
& \left.
\begin{aligned}
& \vdots \\
& t^{M/2+c}\\
& t^{M/2+c-1}\\
& t^{M/2+c}
\end{aligned}\right.
\end{aligned}\right\rbrace L
\end{aligned}
\end{equation}
Notice that $L$ is not a multiple of $A$, so that \eqref{eq:RG_1} is not written in a symmetric way.  Nevertheless, it is symmetric in the sense that if each generator $t^{\alpha}$ was replaced with $t^{-\alpha}$, the overall set in \eqref{eq:RG_1} would remain unchanged. This can be seen geometrically by considering an ``upside-down" version of Figure \ref{fig:int_rgp}, in which the ``rotation" part of the red curve occurs near the bottom rather than the top.  Then the analogous computation yields \eqref{eq:RG_1} except that the Alexander gradings all have opposite signs.

We are now ready to show the main result for this family of knots (confirming Theorem 1 of \cite{Eft}).

\begin{thm}\label{thm1}
If $K=P(2a,-2b-1,-2c-1)$ with $a,b,c>0$ then for each Alexander grading $s$, $\widehat{HFK}(K,s)$ is supported in at most one delta (equivalently Maslov) grading.
\end{thm}
\begin{proof}
By \eqref{eq:I14_1}, \eqref{eq:I23_1}, \eqref{eq:RN_1}, and \eqref{eq:RG_1}, we see that if we are in Cases \hyperref[c2]{II} or \hyperref[c3]{III}, all of $\widehat{HFK}(K)$ is supported in a single delta grading.  Hence, we consider Case \hyperref[c1]{I} in which $a\leq b$.  Once again, we notice that if $a>c$, the homology is thin and so we turn our attention to the situation where $a \leq c$.  Now $\widehat{HFK}(K)$ is supported in two delta gradings.  \eqref{eq:RP_1} gives us that for $\delta=\frac{3}{2}$, the Alexander gradings of the nonzero elements of $\widehat{HFK}(K)$ lie between $-b+2a-c$ and $b-2a+c$.  On the other hand, by \eqref{eq:I14_1} and \eqref{eq:I23_1}, the Alexander gradings of elements with $\delta = \frac{1}{2}$ lie between $-b-1-c$ and $-b+2a-c-2$ or between $b-2a+c+2$ and $b+1+c$.  Therefore, even in this case, there is no Alexander grading with nonzero elements belonging to more than one delta grading.
\end{proof}

\begin{cor}
For $K=P(2a,-2b-1,-2c-1)$, $\widehat{HFK}(K)$ is determined by the Alexander polynomial of $K$.
\end{cor}

\section{The case $K=P(2a,-2b-1,2c+1)$}
\begin{figure}
\centering
\includegraphics[width=.75\textwidth]{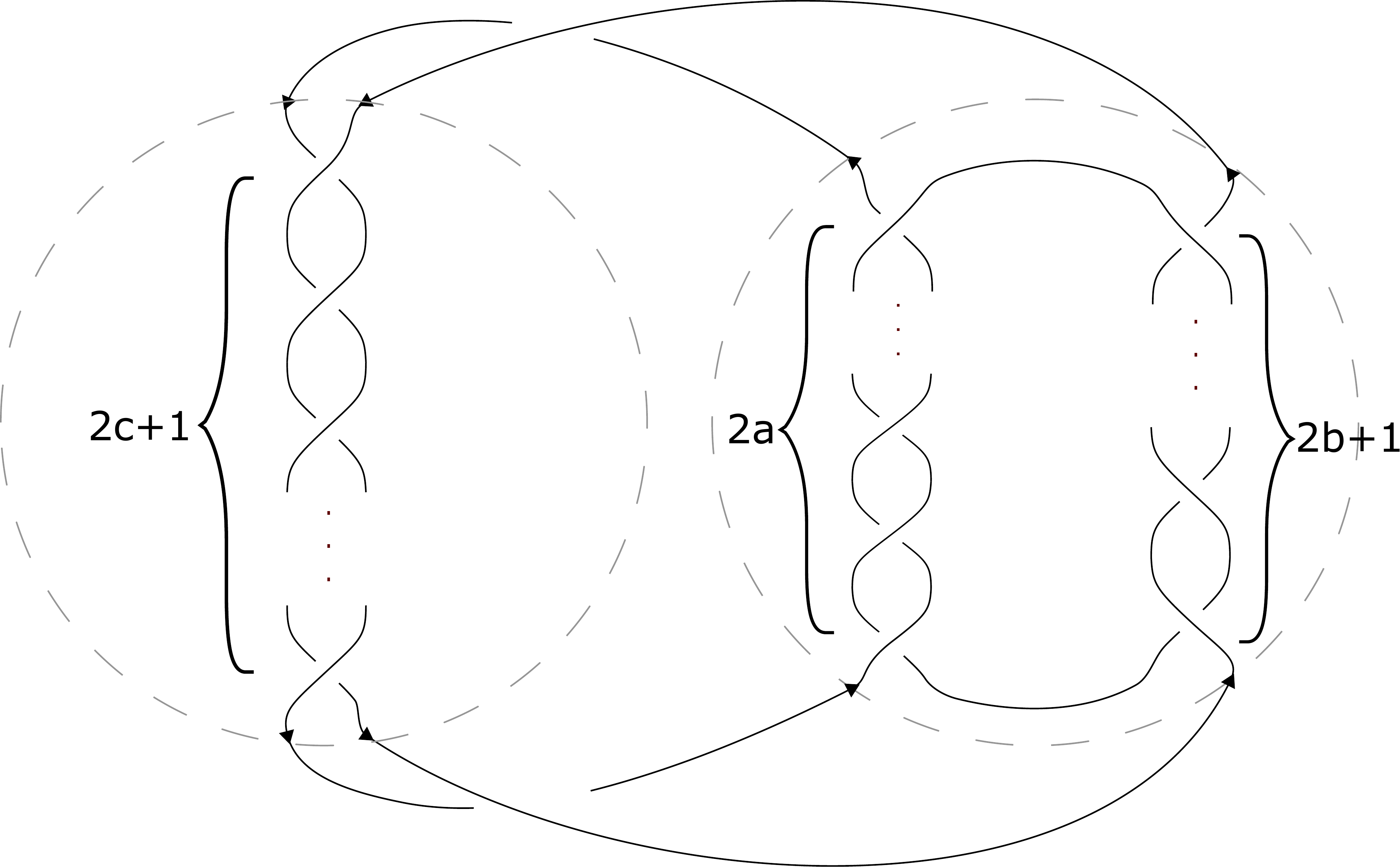}
\caption{The knot $P(2a,-2b-1,2c+1)$ realized as the pairing of two tangles}
\label{fig:pair_P}
\end{figure}
We now turn our attention to the case of the pretzel knot ${P(2a,-2b-1,-2c-1)}$. This corresponds to pairing the $\frac{1}{2c+1}$ rational tangle with the ${(2a,-2b-1)}$-pretzel tangle (see Figure \ref{fig:pair_P}); as in the last section, we need to compute the pairing of curves of the form:
\begin{itemize}
\item $i_k(1,4)t^mt^M$
\item $i_k(2,3)t^mt^M$
\item $r\left(-\frac{1}{2n}\right)t^{-2n}t^{2n}$
\item $r\left(\frac{1}{2n}\right)t^{m}t^{M}$
\item $r\left(-\frac{A}{B}\right)t^{-M}t^{M}$
\end{itemize}
except this time with the rational curve $r\left(-\frac{1}{2c+1}\right)t^{-2c-1}t^{2c+1}$.
\begin{figure}[h]
\includegraphics[width=\textwidth]{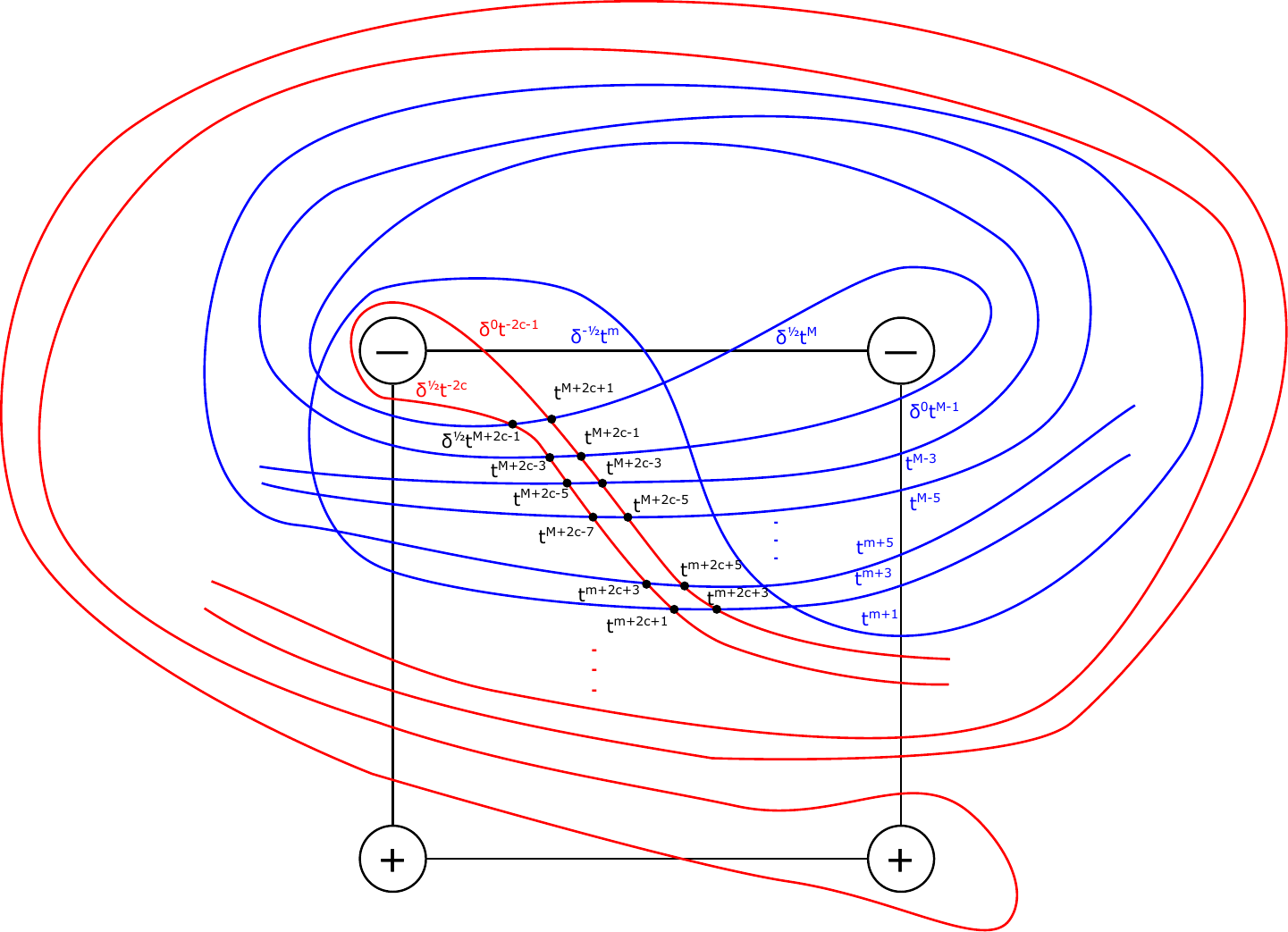}
\caption{The minimal intersection between the curves $i_k(1,4)t^mt^M$ and $r\left(-\frac{1}{2c+1}\right)t^{-2c-1}t^{2c+1}$}
\label{fig:int_i14n}
\end{figure}
Pairing with the curve $i_k(1,4)t^mt^M$ is similar to the analogous computation from the previous section.  As in Figure \ref{fig:int_i14n}, one can arrange that all intersection points occur as the red curve crosses over the blue arcs and avoids the blue curve in all its ``rotations".  By \eqref{eq:int_D} we see that all the intersection points have delta-grading $\frac{1}{2}$.  On the other hand, by \eqref{fig:int_A}, we see that the Alexander gradings are arranged in consecutive pairs:
\begin{align*}
(m+2c+1,m+2c+3), (m+2c+3,m+2c+5), (m+2c+5,m+2c+7), \dots \\
\dots, (M+2c-5,M+2c-3), (M+2c-3,M+2c-1), (M+2c-1,M+2c+1)
\end{align*}
Applying the reduction gives:
\begin{equation}\label{eq:I14_2}
\begin{aligned}
r\left(-\frac{1}{2c+1}\right)t^{-(2c+1)}t^{2c+1} \boxtimes i_k(1,4)t^mt^M \\
=\delta^{1/2}\{t^{m/2+c+1},t^{m/2+c+2},\dots,t^{M/2+c}\}
\end{aligned}
\end{equation}
\begin{figure}[!h]
\includegraphics[width=\textwidth]{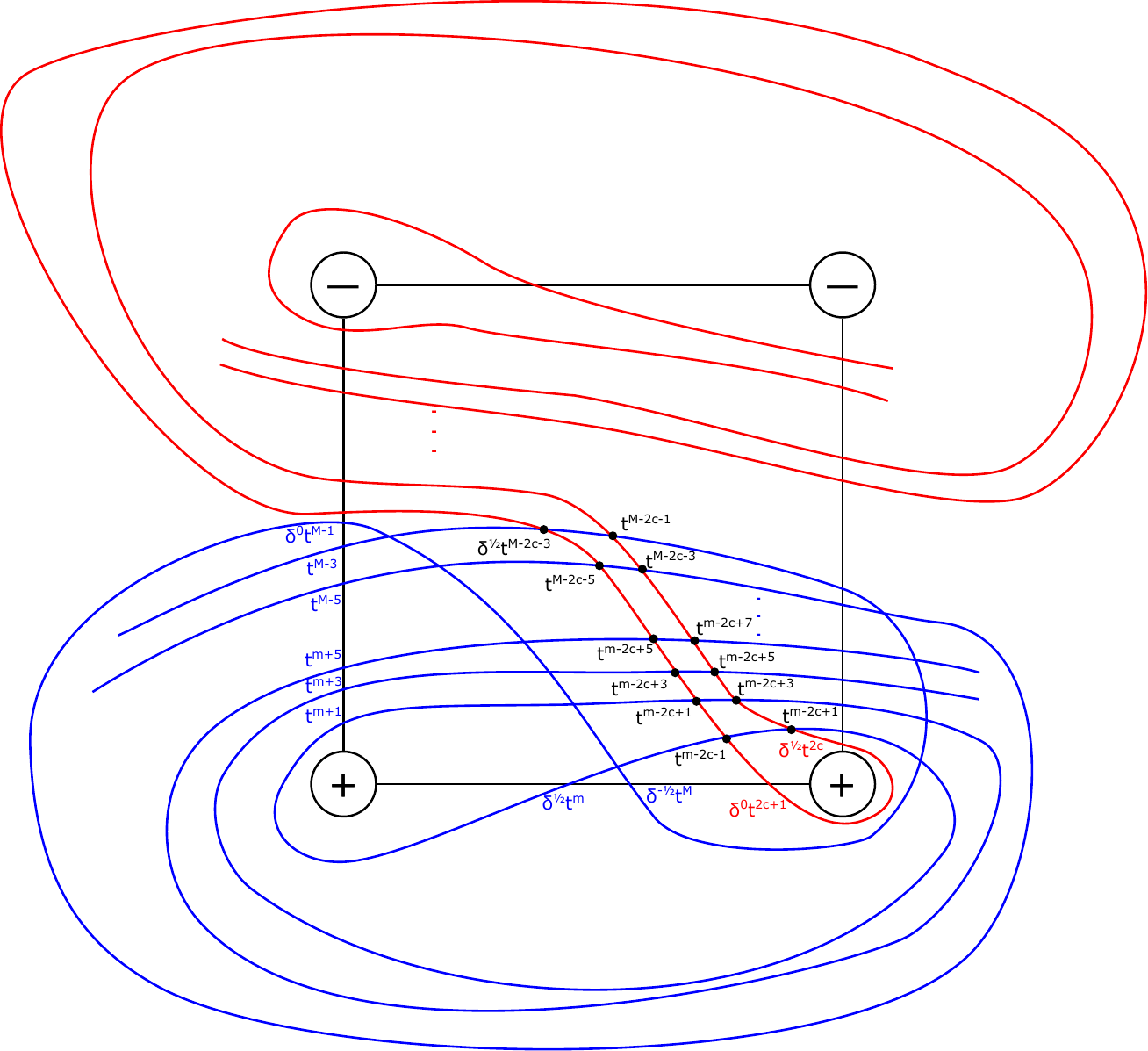}
\caption{The minimal intersection between the curves $i_k(2,3)t^mt^M$ and $r\left(-\frac{1}{2c+1}\right)t^{-2c-1}t^{2c+1}$}
\label{fig:int_i23n}
\end{figure}
The pairing with $i_k(2,3)t^mt^M$ is similar to the previous pairing.  Once again, observing Figure \ref{fig:int_i23n} we see all generators with delta grading $\frac{1}{2}$ and the Alexander gradings arranged in consecutive pairs, each pair shifted by 2 from the previous.  In particular, the Alexander gradings are:
\begin{align*}
(m-2c-1,m-2c+1),(m-2c+1,m-2c+3),(m-2c+3,m-2c+5),\dots \\
\dots, (M-2c-7,M-2c-5),(M-2c-5,M-2c-3),(M-2c-3,M-2c-1)
\end{align*}
Applying the reduction gives:
\begin{equation}\label{eq:I23_2}
\begin{aligned}
r\left(-\frac{1}{2c+1}\right)t^{-(2c+1)}t^{2c+1} \boxtimes i_k(2,3)t^mt^M \\
=\delta^{1/2}\{t^{m/2-c},t^{m/2-c+1},\dots,t^{M/2-c-1}\}
\end{aligned}
\end{equation}

\begin{figure}[h]
\centering
\begin{subfigure}{\textwidth}
\centering
\includegraphics[width=\textwidth]{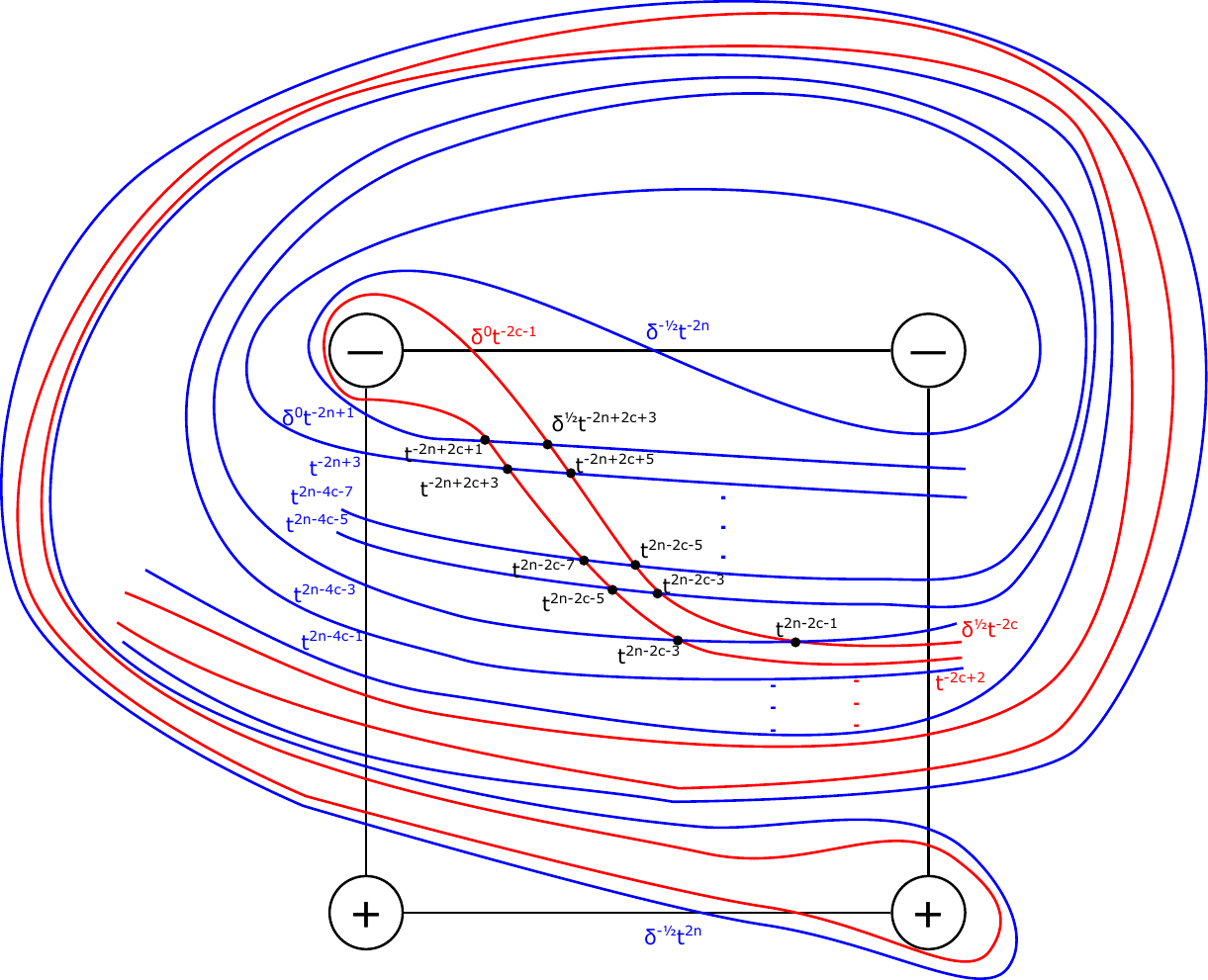}
\caption{The case where $n>c$}
\label{fig:int_rnn1}
\end{subfigure}
\caption{The minimal intersection between the curves $r\left(-\frac{1}{2n}\right)t^{-2n}t^{2n}$ and $r\left(-\frac{1}{2c+1}\right)t^{-2c-1}t^{2c+1}$; the two cases}
\label{fig:int_rnn}
\end{figure}
\begin{figure}\ContinuedFloat
\begin{subfigure}{\textwidth}
\includegraphics[width=\textwidth]{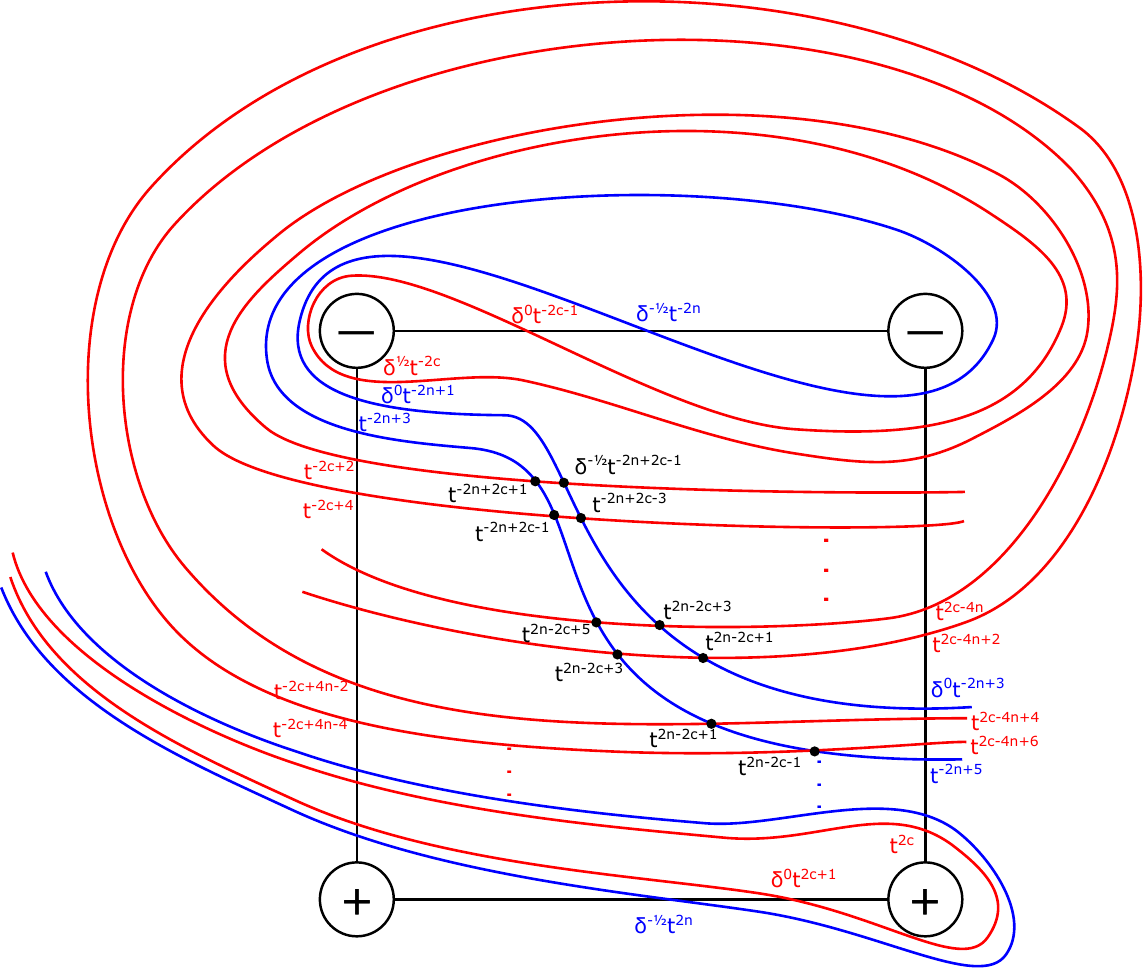}
\caption{The case where $n\leq c$}
\label{fig:int_rnn2}
\end{subfigure}
%\caption{The minimal intersection between the curves $r\left(-\frac{1}{2n}\right)t^{-2n}t^{2n}$ and $r\left(-\frac{1}{2c+1}\right)t^{-2c-1}t^{2c+1}$; the two cases (cont.)}
\end{figure}
We now consider the pairing with the rational curve $r\left(-\frac{1}{2n}\right)t^{-2n}t^{2n}$.  This is essentially the mirror image of the pairing $HF\left(r\left(\frac{1}{2c+1}\right)t^{-(2c+1)}t^{2c+1}, r\left(\frac{1}{2n}\right)t^{m}t^{M}\right)$ from the previous section, with $M=2n$ (see Figure \ref{fig:int_rnn} vis-\`{a}-vis Figure \ref{fig:int_rpp}).  As in that case, there are again two subcases depending on whether $n>c$ or $n\leq c$.  In the first case, by Figure \ref{fig:int_rnn1}, we see that the intersection points have delta-grading $\frac{1}{2}$ (as usual, we only need to compute this for one such point, in light of Lemma \ref{lem:rat_delta}) and Alexander gradings occurring in consecutive pairs:
\begin{align*}
(-2n+2c+1,-2n+2c+3), (-2n+2c+3,-2n+2c+5), (-2n+2c+5,-2n+2c+7), \dots \\
\dots, (2n-2c-7,2n-2c-5), (2n-2c-5,2n-2c-3), (2n-2c-3,2n-2c-1)
\end{align*}
In the second case, by Figure \ref{fig:int_rnn2}, we see that the intersection points have delta-grading $-\frac{1}{2}$, with Alexander gradings occurring in consecutive pairs:
\begin{align*}
(2n-2c-1,2n-2c+1), (2n-2c+1,2n-2c+3), (2n-2c+3,2n-2c+5), \dots \\
\dots, (-2n+2c-3,-2n+2c-5), (-2n+2c-1,-2n+2c-3), (-2n+2c+1,-2n+2c-1)
\end{align*}
After applying the reduction, we obtain:
\begin{equation}\label{eq:RN_2}
\begin{aligned}
r\left(-\frac{1}{2c+1}\right)t^{-(2c+1)}t^{2c+1} \boxtimes r\left(-\frac{1}{2n}\right)t^{-2n}t^{2n} \\
=\begin{cases}
\delta^{1/2}\{t^{-n+c+1},t^{-n+c+2},\dots,t^{n-c-1}\} & \textrm{ if }n>c \\
\delta^{-1/2}\{t^{n-c},t^{n-c+1},\dots,t^{c-n}\} & \textrm{ if }n\leq c
\end{cases}
\end{aligned}
\end{equation}
Once again, one can also see this alternatively as the Knot Floer homology of the $(2,2c-2n+1)$-torus knot.

\begin{figure}[h]
\includegraphics[width=\textwidth]{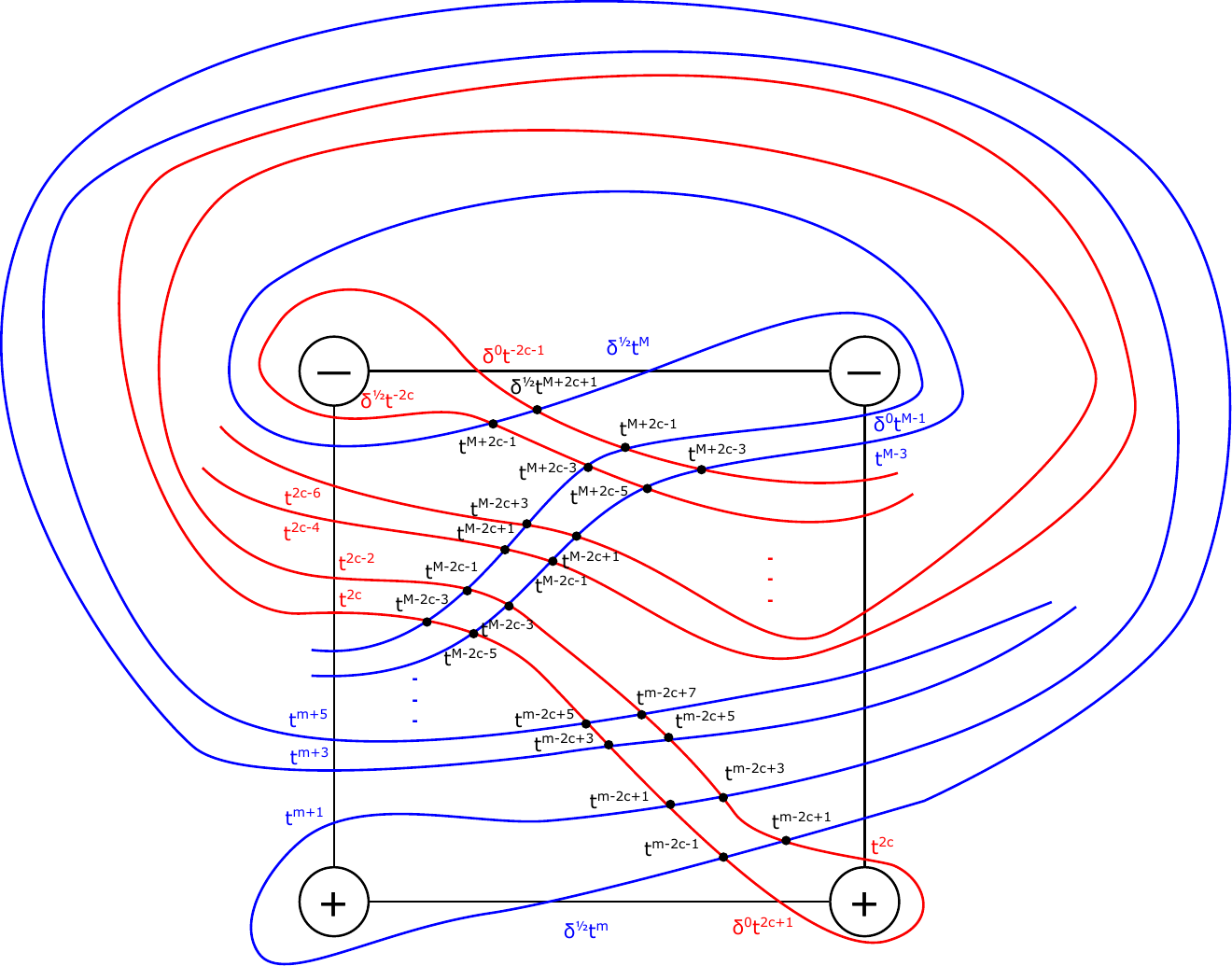}
\caption{The minimal intersection between the curves $r\left(\frac{1}{2n}\right)t^{m}t^{M}$ and $r\left(-\frac{1}{2c+1}\right)t^{-2c-1}t^{2c+1}$}
\label{fig:int_rpn}
\end{figure}
We now consider the pairing with the rational curve $r\left(-\frac{1}{2n}\right)t^{-2n}t^{2n}$.  This is essentally the mirror image of the pairing $HF\left(r\left(\frac{1}{2c+1}\right)t^{-(2c+1)}t^{2c+1}, r\left(-\frac{1}{2n}\right)t^{-2n}t^{2n}\right)$ from the previous section, except this time the blue curve has a more general Alexander grading (see Figure \ref{fig:int_rpn} and compare with Figure \ref{fig:int_rnp}).  From Figure \ref{fig:int_rpn} we see that the intersection points have delta-grading $\frac{1}{2}$ (as usual, we only need to compute this for one point by Lemma \ref{lem:rat_delta}).  Moreover, the Alexander gradings are arranged in consecutive pairs:
\begin{align*}
(m-2c-1,m-2c+1),(m-2c+1,m-2c+3),(m-2c+3,m-2c+5), \dots \\
\dots,(M+2c-5,M+2c-3),(M+2c-3,M+2c-1),(M+2c-1,M+2c+1)
\end{align*}
Applying the reduction, we obtain:
\begin{equation}\label{eq:RP_2}
\begin{aligned}
r\left(-\frac{1}{2c+1}\right)t^{-(2c+1)}t^{2c+1} \boxtimes r\left(\frac{1}{2n}\right)t^{m}t^{M} \\
\delta^{1/2}\{t^{m/2-c},t^{m/2-c+1},\dots,t^{M/2+c}\}
\end{aligned}
\end{equation}
As before, this can also be seen to be (up to an overall grading shift) the Knot Floer homology of the $(2,2n+2c+1)$-torus knot.

\begin{figure}[!h]
\centering
\begin{subfigure}{\textwidth}
\centering
\includegraphics[width=\textwidth]{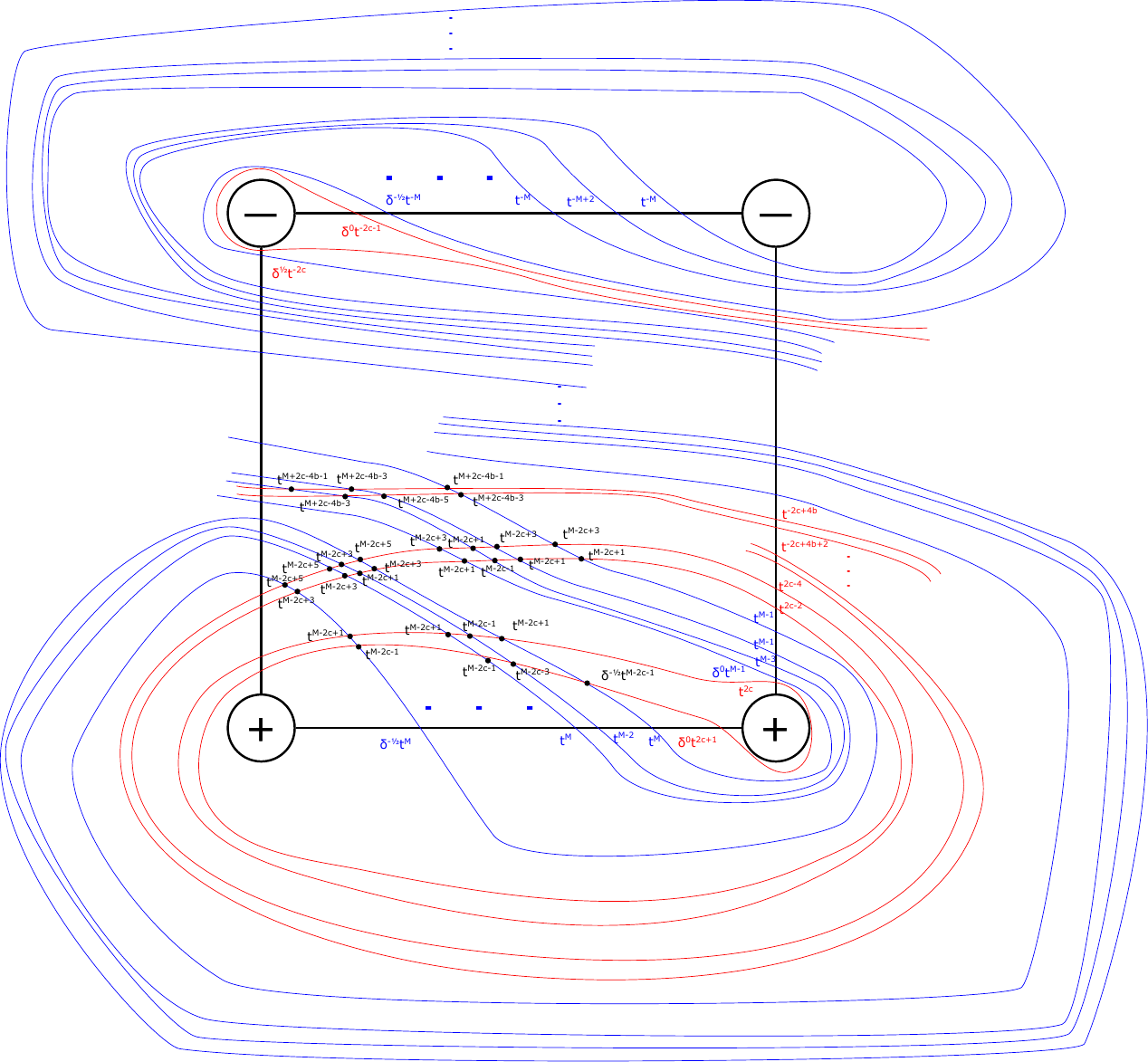}
\caption{The case where $\frac{A}{B}>\frac{1}{2c+1}$}
\label{fig:int_rgn1}
\end{subfigure}
\caption{The minimal intersection between the curves $r\left(-\frac{A}{B}\right)t^{-M}t^{M}$ and $r\left(-\frac{1}{2c+1}\right)t^{-2c-1}t^{2c+1}$; the two cases}
\label{fig:int_rgn}
\end{figure}
\begin{figure}\ContinuedFloat
\begin{subfigure}{\textwidth}
\includegraphics[width=\textwidth]{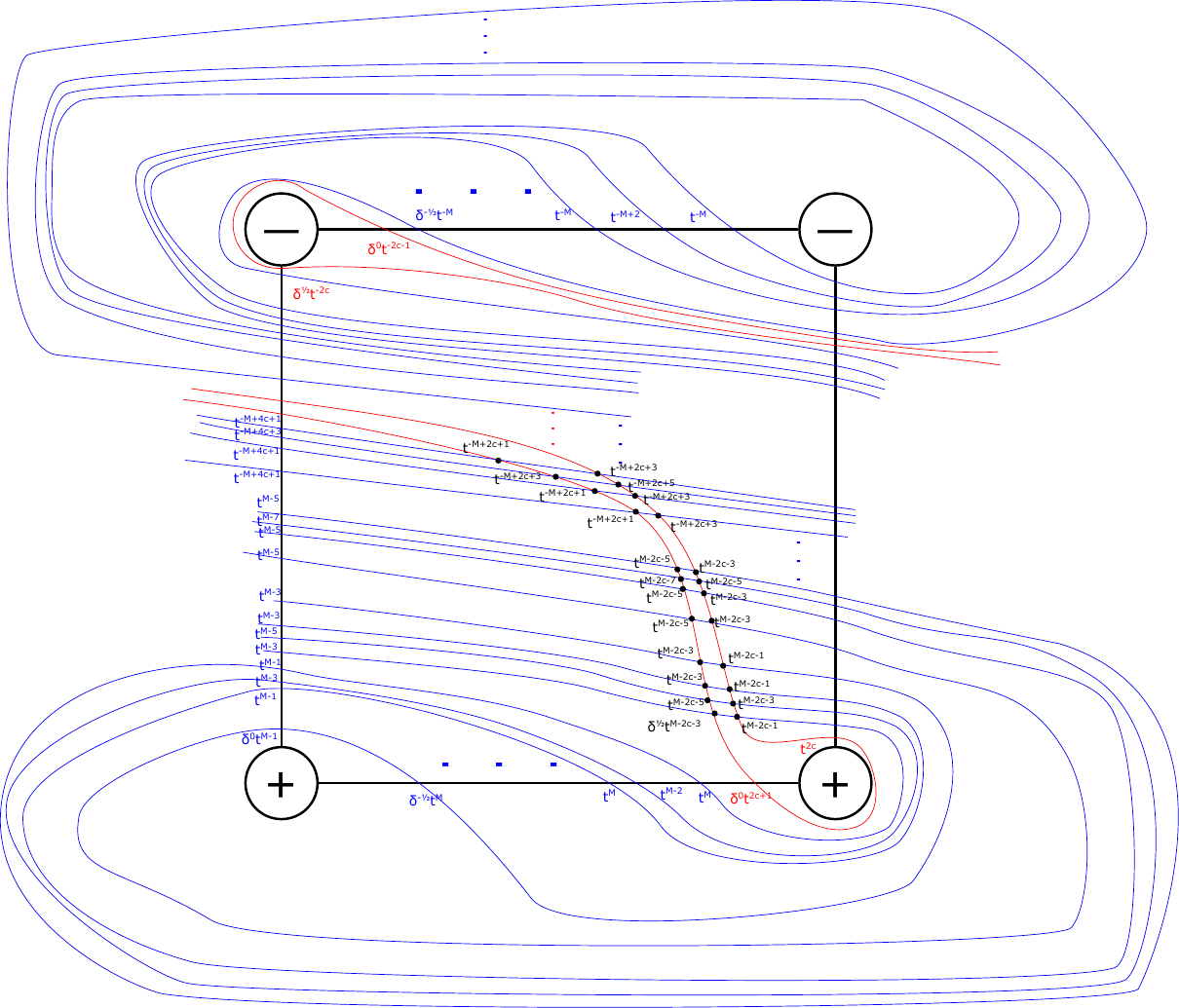}
\caption{The case where $\frac{A}{B}<\frac{1}{2c+1}$}
\label{fig:int_rgn2}
\end{subfigure}
%\caption{The minimal intersection between the curves $r\left(-\frac{A}{B}\right)t^{-M}t^{M}$ and $r\left(-\frac{1}{2c+1}\right)t^{-2c-1}t^{2c+1}$; the two cases (cont.)}
\end{figure}
We now turn our attention to the pairing with the more general rational curve $r\left(-\frac{A}{B}\right)t^{-M}t^M$, where, by Case \hyperref[c3]{III}, we have $A=2(a-b)-1$, $B=4b(a-b-1)+2a=A(2b+1)+1$, and $M=2b+2$.  As this is a pairing between rational curves with slopes of the same sign, we consider two separate cases: $\frac{A}{B}>\frac{1}{2c+1}$ and $\frac{A}{B}<\frac{1}{2c+1}$.  In the first case, as we see in Figure \ref{fig:int_rgn1}, we can arrange the curves so that the red curve, starting from the top-left corner, follows the blue curve for $b$ ``rotations".  Since $B=A(b+1)+1$, after $b$ rotations, the parallel red arcs intersect the left edge of the square at a point where there are exactly $A+1$ blue intersections below it.  Then, we can let the red arcs cross over $A-1$ blue arcs as in the figure, before continuing the rest of its ``rotations" in such a way that the intersections only occur as the red arcs cross the lower $2A$ blue arcs until crossing the final $A$ arcs before reaching the bottom-right puncture.

From this configuration, we see that all intersection points have delta-grading $-\frac{1}{2}$ (as usual, this only needs to be checked once) and the Alexander gradings are arranged in groups of $A$ pairs of the form:
\begin{align*}
(\alpha-1,\alpha+1),(\alpha-1,\alpha-3),(\alpha-1,\alpha+1),(\alpha-1,\alpha-3),(\alpha-1,\alpha+1)...\\
...(\alpha-1,\alpha+1),(\alpha-1,\alpha-3),(\alpha-1,\alpha+1)
\end{align*}
for $\alpha = M-2c, M-2c+2, M-2c+4,\dots, M+2c-4b-2$.  Since $M=2b+2$, this last value may be written as: $-M+2c+2$ Notice that for this largest value of $\alpha$, there are $A-1$ pairs as the last pair is omitted.  This is consistent with Lemma \ref{lem:rat_det} which predicts $L=\det\left(\begin{matrix}
A & -1\\
-B & 2c+1
\end{matrix}\right) = A(2c+1)-B=2A(c-b)-1$.
Notice also that, by turning Figure \ref{fig:int_rgn1} upside-down, one sees that the multiset of Alexander gradings is invariant under multiplication by $-1$.
In the second case ($\frac{A}{B}<\frac{1}{2c+1}$), we see that, as in Figure \ref{fig:int_rgn2}, the curves may be arranged such that the red arcs, starting at the top-left puncture follow the blue curve for all $c$ of their ``rotations" before crossing over the rest of the blue arcs to reach the bottom-right puncture.  We see that in this case, the intersection points have delta-grading $\frac{1}{2}$.  Similarly as before, the Alexander gradings are arranged in groups of $A$ pairs of the form
\begin{align*}
(\alpha-1,\alpha+1),(\alpha+1,\alpha+3),(\alpha-1,\alpha+1),(\alpha+1,\alpha+3),(\alpha-1,\alpha+1)...\\
...(\alpha-1,\alpha+1),(\alpha+1,\alpha+3),(\alpha-1,\alpha+1)
\end{align*}
for $\alpha = -M+2c+2, -M+2c+4, -M+2c+6, \dots$.  By Lemma \ref{lem:rat_det}, we know that we can take the first $L=\det\left(\begin{matrix}
-A & -1\\
B & 2c+1
\end{matrix}\right) = B-A(2c+1)$ pairs in this sequence.  Notice, moreover, that by starting at the bottom of Figure \ref{fig:int_rgn2} and moving upwards, one can instead take the sequence of groups of pairs of the form:
\begin{align*}
(\alpha-1,\alpha+1),(\alpha-1,\alpha-3),(\alpha-1,\alpha+1),(\alpha-1,\alpha-3),(\alpha-1,\alpha+1)...\\
...(\alpha-1,\alpha+1),(\alpha-1,\alpha-3),(\alpha-1,\alpha+1)
\end{align*}
with $\alpha = M-2c-2, M-2c-4, M-2c-6,\dots$.  Thus, the (multi)set of Alexander gradings is invariant under multiplication by $-1$.

Combining the two cases and applying the reduction, we obtain:
\begin{equation}\label{eq:RG_2}
\begin{aligned}
& r\left(-\frac{1}{2c+1}\right)t^{-(2c+1)}t^{2c+1} \boxtimes r\left(-\frac{A}{B}\right)t^{-M}t^{M} \\
&\frac{A}{B}>\frac{1}{2c+1} && \frac{A}{B}<\frac{1}{2c+1} \\
& =\delta^{-1/2} && =\delta^{1/2} \\
& \left.
\begin{aligned}
& \left.
\begin{aligned}
&t^{M/2-c}\\
&t^{M/2-c-1}\\
&t^{M/2-c}\\
&\vdots \\
&t^{M/2-c}
\end{aligned}\right\rbrace A \\
& \left.
\begin{aligned}
& t^{M/2-c+1}\\
& t^{M/2-c}\\
& t^{M/2-c+1}\\
& \vdots \\
& t^{M/2-c+1}
\end{aligned}\right\rbrace A \\
&\left.\begin{aligned}
& \vdots \\
& t^{-M/2+c}\\
& t^{-M/2+c+1}\\
& t^{-M/2+c}
\end{aligned}\right.
\end{aligned}
\right\rbrace L && 
\left.
\begin{aligned}
& \left.
\begin{aligned}
& t^{-M/2+c+1}\\
& t^{-M/2+c+2}\\
& t^{-M/2+c+1}\\
& \vdots \\
& t^{-M/2+c+1}
\end{aligned}\right\rbrace A \\
&\left.
\begin{aligned}
& t^{-M/2+c+2}\\
& t^{-M/2+c+3}\\
& t^{-M/2+c+2}\\
&\vdots \\
& t^{-M/2+c+2}
\end{aligned}\right\rbrace A \\
& \left. \begin{aligned}
& \vdots \\
& t^{M/2-c-1}\\
& t^{M/2-c-2}\\
& t^{M/2-c-1}
\end{aligned}\right.
\end{aligned}
\right\rbrace L
\end{aligned}
\end{equation}

We are now able to prove our main results regarding the Knot Floer homolog of the knot $P(2a,-2b-1,2c+1)$.
\begin{lem}
If $K=P(2a,-2b-1,2c+1))$ with $a,b,c>0$ and if moreover $a\leq b+1$, then for each Alexander grading $s$, $\widehat{HFK}(K,s)$ is supported in at most one delta (equivalently Maslov) grading.
\end{lem}
\begin{proof}
By hypothesis, as regards the pretzel tangle, we are in Cases \hyperref[c1]{I} and \hyperref[c2]{II}.  In the first case, by \eqref{eq:I14_2}, \eqref{eq:I23_2}, and \eqref{eq:RP_2}, we have that all generators are in a single delta grading, and so our claim follows.  In Case \hyperref[c2]{II}, once again we see that if $a>c$ all generators are in the same delta grading, so we may assume $a\leq c$.  Now by \eqref{eq:RN_2}, the generators with delta grading $-\frac{1}{2}$ have Alexander gradings supported in the interval $[a-c,c-a]$.  By \eqref{eq:I14_2} and \eqref{eq:I23_2}, we see that the generators with delta grading $\frac{1}{2}$ have Alexander grading at least $c-a+1$ or at most $a-c-1$.  Hence, there is no overlap between the two types of delta grading.
\end{proof}

In light of this, we turn our attention to Case \hyperref[c3]{III}, that is the case where $a>b+1$.  According to \eqref{eq:RG_2}, $\widehat{HFK}$ will be supported in a single delta grading unless $\frac{A}{B}>\frac{1}{2c+1}$.  Recall that $A=2(a-b)-1$ and $B=4b(a-b-1)+2a = A(2b+1)+1$.  We have:
%\begin{align*}
%\frac{A}{B} &>\frac{1}{2c+1}\\
%\Leftrightarrow \frac{2\omega+1}{4b\omega+2a} &> \frac{1}{2c+1}\\
%\Leftrightarrow 4c\omega +2\omega +2c+1 &> 4b\omega +2a\\
%\Leftrightarrow 4\omega(c-b) &> 2(a-c-\omega) -1\\
%\Leftrightarrow 4\omega(c-b) &> 2(b-c)+1\\
%\Leftrightarrow 2(2\omega+1)(c-b) &> 1\\
%\Leftrightarrow c &> b
%\end{align*}
\begin{align*}
\frac{A}{B} &>\frac{1}{2c+1}\\
\Leftrightarrow \frac{A}{A(2b+1)+1} &> \frac{1}{2c+1}\\
\Leftrightarrow A(2c+1) &> A(2b+1)+1\\
\Leftrightarrow 2A(c-b) &> 1\\
\Leftrightarrow c &> b
\end{align*}
Hence, the interesting case occurs when $b<\min(c,a-1)$ and, by \eqref{eq:RG_2}, the generators with delta grading $-\frac{1}{2}$ all have Alexander gradings within the range $[b-c,c-b]$.  On the other hand, by \eqref{eq:I14_2} and \eqref{eq:I23_2}, the generators with delta grading $\frac{1}{2}$ either have Alexander grading at least $c-b$ or at most $b-c$.  So there is overlap precisely at these Alexander gradings.  We may also compute the respective ranks: in the case of $\delta = \frac{1}{2}$, we see that only curves of the form $i_k(1,4)t^{2b-2}t^M$ contribute any (in fact exactly one) generator of Alexander grading $c-b$, and there are $b$ such curves.  Similarly for grading $b-c$.  In the case of $\delta = -\frac{1}{2}$, we see that by \eqref{eq:RG_2} there are $\frac{A-1}{2} = a-b-1$ generators of Alexander grading $b-c$ and symmetrically for $c-b$.  These results are summarized in Table \ref{tab:rk}.
\begin{table}[h]
\centering
\begin{tabular}{|c|c|c|}
\hline
& $s=b-c$ & $s=c-b$
\\
\hline
$\delta= 1/2$ & $b$ & $b$\\
\hline
$\delta= -1/2$ & $a-b-1$ & $a-b-1$\\
\hline
\end{tabular}
\caption{The ranks of $\widehat{HFK}_{\delta}(K,s)$} in the ``overlap" region in the case $b<\min(a-1,c)$
\label{tab:rk}
\end{table}

Thus we have shown
\begin{thm}\label{thm2}
If $K=P(2a,-2b-1,2c+1)$ with $a,b,c>0$ and ${b\geq \min(a-1,c)}$ then for each Alexander grading $s$, $\widehat{HFK}(K,s)$ is supported in at most one delta (equivalently Maslov) grading.  Otherwise, this is true for each Alexander grading except for $s=\pm(c-b)$ in which case there are two delta gradings for which $\widehat{HFK}(K,s)$ is nonzero, and the ranks are as in Table \ref{tab:rk}.
\end{thm}
It follows that, in the case where ${b\geq \min(a-1,c)}$, the Knot Floer homology of this family of knots is determined by the Alexander polynomial, as predicted by Theorem 2 of \cite{Eft}.  This does not hold whenever $b<\min(a-1,c)$, and our work corrects Eftekhary's result in that case.

\section*{Acknowledgments}
The author would like to thank Zolt\'{a}n Szab\'{o} for encouraging work on this project. The author also thanks Andy Manion and Jonathan Hanselman for helpful discussions as well as Claudius Zibrowius for comments on a draft of this paper.  This work was supported by NSF grants DMS-1502424 and DMS-1606571.

\printbibliography

\end{document}